\documentclass[reqno,11pt]{amsart}
\usepackage{amsmath}
\usepackage{stmaryrd}
\usepackage{amsfonts}
\usepackage{amsthm}
\usepackage{amssymb}
\usepackage{graphicx,psfrag,epsfig}
\usepackage{color}
\usepackage{mathrsfs}
\usepackage{pdfsync}
\usepackage{cite}
\usepackage{times}
\usepackage{cancel}
\newtheorem{thm}{Theorem}[section]
\newtheorem{lem}[thm]{Lemma}
\newtheorem{cor}[thm]{Corollary}
\newtheorem{prop}[thm]{Proposition}

\newtheorem{rem}[thm]{Remark}

\newtheorem{definition}[thm]{Definition}
\newtheorem*{definition*}{Definition}

\DeclareMathAlphabet{\mathpzc}{OT1}{pzc}{m}{it}

\numberwithin{equation}{section}

\newcommand{\R}{\mathbb{R}}

\newcommand{\N}{\mathbb{N}}

\newcommand{\ve}{\varepsilon}

\newcommand{\rd}{\mathrm{d}}

\newcommand{\tred}{\color{red}}

\newcommand{\dhr}{\mathrel{\lhook\joinrel\relbar\kern-.8ex\joinrel\lhook\joinrel\rightarrow}} 
\title[A Constrained Model for MEMS with Varying Dielectric Properties]
{A Constrained Model for MEMS with Varying Dielectric Properties}
\author{Philippe Lauren\c{c}ot}
\thanks{Partially supported by the CNRS project~PICS07710}
\address{Institut de Math\'ematiques de Toulouse, UMR~5219, Universit\'e de Toulouse, CNRS \\ F--31062 Toulouse Cedex 9, France}
\email{laurenco@math.univ-toulouse.fr}

\author{Christoph Walker}
\address{Leibniz Universit\"at Hannover\\ Institut f\" ur Angewandte Mathematik \\ Welfengarten 1 \\ D--30167 Hannover\\ Germany}
\email{walker@ifam.uni-hannover.de}

\date{\today}

\begin{document}

\begin{abstract}
A semilinear parabolic equation with constraint modeling the dynamics of a microelectromechanical system (MEMS) is studied. In contrast to the commonly used MEMS model, the well-known pull-in phenomenon occurring above a critical potential threshold is not accompanied by a break-down of the model, but is recovered by the saturation of the constraint for pulled-in states. It is shown that a maximal stationary solution exists and that saturation only occurs for large potential values.
In addition, the existence, uniqueness, and large time behavior of solutions to the evolution equation are studied.
\end{abstract}


\maketitle

\section{Introduction}

We investigate the well-posedness and qualitative behavior of solutions to the following equation 
\begin{subequations}\label{EvEq001x}
\begin{align}
\partial_t u - \Delta u + \partial\mathbb{I}_{[-1,\infty)}(u) & \owns - \frac{\lambda}{2(1+u+W(x))^2} \,,\qquad t>0\,,\quad x\in  D\, , \label{EvEq001ax} \\
u & = 0 \,,\qquad t>0\,,\quad x\in  \partial D\, ,\label{EvEq001cx} \\
u(0) & = u_0 \,,\qquad   x\in  D\, , \label{EvEq001dx}
\end{align}
\end{subequations}
arising from the modeling of idealized electrostatically actuated  microelectromechanical systems  (MEMS) with varying dielectric properties. Here, $D$ is the shape at rest of  membrane coated with a thin dielectric layer which is held fixed  on its boundary and suspended  above a rigid horizontal ground plate with the same shape $D$. Holding the ground plate at potential zero and applying a positive potential to  the membrane induce a Coulomb force across the device and thereby a deformation of the membrane. After a suitable rescaling, the ground plate is located at vertical position $z=-1$ while the membrane at rest is located at $z=0$, and its vertical deflection  $u(t,x)$ at time $t\ge 0$ and position $x\in D$ solves \eqref{EvEq001x}. The parameter $\lambda$ in \eqref{EvEq001ax} is proportional to the applied voltage while $W$ is non-negative and depends on the spatial position $x$ and accounts for the possible dielectric heterogeneity of  the membrane. We point out that inertia and bending effects are neglected in \eqref{EvEq001x}. This model is derived in \cite{LW17} where we revisit the derivation of MEMS models with varying dielectric properties and differs from the commonly used model to describe the dynamics of MEMS which reads \cite{Pe02}
\begin{subequations}\label{Pe}
\begin{align}
\partial_t u - \Delta u  & = - \frac{\lambda f(x)}{2(1+u)^2} \,,\qquad t>0\,,\quad x\in  D\, , \label{Pea} \\
u & = 0 \,,\qquad t>0\,,\quad x\in  \partial D\, ,\label{Pec} \\
u(0) & = u_0 \,,\qquad   x\in  D\, , \label{Ped}
\end{align}
\end{subequations}
where $u$ and $\lambda$ have the same meaning as above, but the dielectric properties of the membrane are accounted for by the function $f$ which is non-negative and depends on the spatial position $x\in D$. The difference in the reaction terms in  \eqref{EvEq001ax} and \eqref{Pea}  stems from different approaches to compute the electrostatic force exerted on  the  membrane in the modeling, and we refer to \cite{LW17} and \cite{Pe02} for the complete derivations. Also, the thickness of  the  membrane with heterogeneous dielectric properties is retained when deriving \eqref{EvEq001ax}.

From a physical point of view an ubiquitous feature of MEMS devices is that when the applied potential exceeds a certain threshold value the restoring elastic forces no longer balance the electrostatic forces, and  the  membrane touches down on the ground plate, a phenomenon known as pull-in instability \cite{PeB03}. From the mathematical point of view this means that when $\lambda$ is larger than a certain threshold value $\lambda_*$, the diffusion term no longer overcomes the reaction term and there is a time $T_*>0$ such that $\min u(T_*)=-1$. When this occurs, the two models respond in a completely different way. Indeed, in \eqref{Pe} the reaction term becomes singular and the solution ceases to exist at this time (such a behavior is also referred to as {\it quenching} in literature). In contrast, the constraint term $ \partial\mathbb{I}_{[-1,\infty)}(u)$ accounts for the fact that  the  membrane cannot penetrate the ground plate upon touching down but rather lies directly on it. The notation $\partial\mathbb{I}_{[-1,\infty)}(u)$ stands for the subdifferential of the indicator function $\mathbb{I}_{[-1,\infty)}$ of the closed convex set $[-1,\infty)$, the indicator function taking the value zero on $[-1,\infty)$ and the value $\infty$ on its complement. Since $\partial\mathbb{I}_{[-1,\infty)}$ is a set-valued operator, see \eqref{mm} below, the evolution equation  \eqref{EvEq001ax} is actually a differential inclusion which could also be written as a parabolic variational inequality, see for instance \cite{Ba10, Br73}. Owing to this constraint, the evolution equation  \eqref{EvEq001ax} features no singularity, not even in the coincidence region where $u=-1$, at least if $W>0$ in $D$. Therefore, one expects to have global solutions for this model. As we shall prove below, this is indeed true  and, in fact, $W$ may even vanish, but only at isolated points and not to rapidly, the latter being measured by some integrability assumption on $1/W$, see \eqref{W1} below. We shall not explore the influence of a non-empty zero set of $W$ in great detail herein. Since $W$ is proportional to $1/\sigma$, where $\sigma$ denotes the dielectric permittivity of  the  membrane (see \cite{LW17}), the assumption $W>0$ corresponds to a membrane with no perfectly conducting part.

Another striking difference between the two models is that there is no stationary solution to \eqref{Pe} when $\lambda$ exceeds the critical value $\lambda_*$ while there is always at least one stationary solution to \eqref{EvEq001x} for all values of $\lambda$. Nevertheless, as we shall see, there is still a critical value $\Lambda_z>0$ for $\lambda$ which separates stationary solutions in {\it unzipped states} (for $\lambda<\Lambda_z$) and in {\it zipped states} (for $\lambda>\Lambda_z$) defined as:

\begin{definition*}
 A measurable function $h:D\rightarrow [-1,\infty)$ is a {\it zipped state} if the {\it coincidence set} $$\mathcal{C}(h):=\{ x\in D\,;\, h(x)=-1\}$$ has a positive Lebesgue measure and an {\it unzipped state} otherwise.
\end{definition*}

Thus, the issue of non-existence of stationary solutions to \eqref{Pe} is replaced in \eqref{EvEq001x} with the existence of zipped states. Since the pioneering works \cite{GPW05,Pe02,FMPS06,BGP00} a lot of research has been devoted to \eqref{Pe} providing a wealth of information on the structure of stationary solutions, the occurrence of touchdown in finite time, and the dynamical properties of solutions. We refer to \cite{EGG10} and \cite{LWBible} for a more detailed description and references.

Returning to \eqref{EvEq001x}, which is the focus of this paper, let us mention that an equation with a similar constraint  is considered in \cite{GB01} in a (fourth-order stationary) MEMS model with a dielectric layer placed on top of the ground plate. We also refer to \cite{LLG14,LLG15}, where a regularizing term is added in \eqref{Pe}  in order to describe the behavior of a MEMS after initial contact of the membrane and the ground plate.

The purpose of this paper is to provide various results for \eqref{EvEq001x} including a description of the stationary solutions, the well-posedness of the evolution problem as well as qualitative properties of the solutions. These results are presented in the next section.

\section{Main Results}\label{SecMR}

We assume throughout this paper that $D$ is a bounded domain in $\mathbb{R}^d$, $d\ge 1$, with smooth boundary $\partial D$ and that
\begin{equation}\label{W1}
\text{$W$ is a non-negative measurable function on $D$ such that $1/W\in L_2(D)$\,.}
\end{equation}
Further assumptions on $W$ will be explicitly stated later on whenever needed. Let us point out that \eqref{W1} is the minimal assumption to ensure that the right-hand side of \eqref{EvEq001ax} belongs to $L_1(D)$.

\subsection{Stationary Problem}

We shall first present our main results with respect to stationary solutions to \eqref{EvEq001x}. To have a compacter notion of the right-hand side of \eqref{EvEq001x} in the following, we introduce
\begin{equation}\label{g}
g_W(v)(x):= \frac{1}{2 \left(1+v(x)+W(x)\right)^{2}} \,,\quad x\in D\,,
\end{equation}
for a given function $v:D\rightarrow [-1,\infty)$.
We let $-\Delta_1$ be the $L_1$-realization of the Laplace-Dirichlet operator, that is,
$$
-\Delta_1 u:=-\Delta u\,,\qquad u\in D(\Delta_1):=\{w\in W_1^1(D)\,;\, \Delta w\in L_1(D)\,,\ w=0  \text{ on } \partial D\}\,,
$$
where $w=0  \text{ on } \partial D$ is to be understood in the sense of traces. Recall that $D(\Delta_1)$ embeds continuously in $W_q^1(D)$ for $1\le q < d/(d-1)$. Let us also recall that $\partial \mathbb{I}_{[-1,\infty)}$ is the maximal monotone graph in $\R\times\R$  given by
\begin{equation}\label{mm}
\partial \mathbb{I}_{[-1,\infty)}(r)=\left\{ \begin{array}{cl}
\emptyset \,, & r<-1\,,\\
(-\infty,0]\,, & r=-1\,,\\
\{0\}\,, & r>-1\,.
\end{array}\right.
\end{equation}
The following definition gives a precise notion of a stationary solution.

\begin{definition}\label{SSDef001}
A {\it stationary solution} to \eqref{EvEq001x} is a function $u\in D(\Delta_1)$ such that $g_W(u)\in L_1(D)$ and
$$
- \Delta u +\partial \mathbb{I}_{[-1,\infty)}(u)\owns   -\lambda g_W(u) \quad \text{in } D\,.
$$
Equivalently, for a.e. $x\in D$,
$$
\big(\Delta u(x)-\lambda g_W(u)(x)\big) \big(r-u(x)\big)\le 0\,,\quad r\ge -1\,.
$$
\end{definition}

Owing to the integrability of $\Delta u$ and $g_W(u)$, the differential inclusion in Definition~\ref{SSDef001} is to be understood in $L_1(D)$, that is, for a.e. $x\in D$. Throughout the paper we shall omit ``a.e.'' when no confusion seems likely.

The main result regarding stationary solutions is the following.

\begin{thm}[{\bf Maximal Stationary Solutions}]\label{statsol}
Suppose \eqref{W1}. Given $\lambda>0$, there is a maximal stationary solution $U_\lambda \in \mathring{H}^1(D)\cap  D(\Delta_1)$ to \eqref{EvEq001x} with $-1\le U_\lambda\le 0$ in $D$, and there is $\Lambda_z\in (0,\infty)$ such that $U_\lambda$ is unzipped for $\lambda<\Lambda_z$ and zipped for $\lambda>\Lambda_z$.
Moreover, $U_\lambda$ is decreasing with respect to $\lambda$ in the sense that if $\lambda_1<\lambda_2$, then $U_{\lambda_1}\ge U_{\lambda_2}$ in $D$. 

Finally, if $1/W\in L_{2p}(D)$ for some $p\in (1,\infty)$, then $U_\lambda\in W_p^2(D)$.
\end{thm}

Interestingly, Theorem~\ref{statsol} guarantees the existence of at least one stationary solution to \eqref{EvEq001x} for {\it any} value of $\lambda$. As already mentioned this markedly contrasts with the commonly used vanishing aspect ratio model \eqref{Pe} for which no stationary solution exists for large values of $\lambda$, see \cite{BGP00, GPW05, Pe02}. Nevertheless, the role of the critical value of $\lambda$ is played by $\Lambda_z$ which separates the structural behavior of stationary solutions.

Theorem~\ref{statsol} is proven in Section~\ref{Sec2} to which we also refer for a precise definition of a maximal stationary solution (see Proposition~\ref{P1}) and for additional information on stationary solutions in general. The existence result is obtained by a rather classical monotone iterative scheme similar to the one used in \cite{EGG10,GhG06} to construct stationary solutions to \eqref{Pe}. However, due to the constraint in \eqref{EvEq001ax}  the proof is based on the analysis of cite{BS73} on semilinear second-order equations featuring maximal monotone graphs in $L_1$.

To complement the investigation of stationary solutions, we consider in Section~\ref{Sec3} the particular case when $d=1$ and $W\equiv const$. For this situation  we  present in Theorem~\ref{T44} a complete characterization of all stationary solutions. Even in this simplified setting, the structure of stationary solutions turns out to be quite sensitive with respect to the value of $\lambda$. In particular, it is shown that if $W$ is small, then there is an interval for $\lambda$ for which there is coexistence of unzipped and zipped states, a feature for which numerical evidence is provided in \cite{GB01} for a related model.

\subsection{Evolution Problem}

We next consider the evolution equation as stated in \eqref{EvEq001x}. Interestingly it can be seen as the gradient flow in $L_2(D)$ associated with the total energy
\begin{equation}\label{E}
\mathcal{E}_W(u):= \frac{1}{2}\int_D\vert\nabla u\vert^2\,\rd x+\int_D \mathbb{I}_{[-1,\infty)}(u)\,\rd x - \frac{\lambda}{2} \int_D\frac{\rd x}{1+u+W}\,.
\end{equation}
However, our analysis relies only partially on this structure since the functional setting we work with is $L_1(D)$ due to the integrability assumption \eqref{W1} on $1/W$. We use the notation
\begin{equation}
\mathcal{A} := \left\{ v \in \mathring{H}^1(D)\ ;\ v \ge -1 \;\text{ a.e. in }\; D \right\} \label{EvEq002}
\end{equation}
for the domain of the convex part of the energy $\mathcal{E}_W$,
where
$$
\mathring{H}^1(D):=\{v\in H^1(D)\,;\, v=0 \text{ on } \partial D\}\,.
$$
For our purpose, the framework of weak solutions turns out to be not sufficient. Thus, we introduce the stronger notion of an \textit{energy solution}.

\begin{definition}\label{EvDef001}
Let $u_0\in \mathcal{A}$. An {\it energy solution} to \eqref{EvEq001x} is a function $u$ such that, for all $t>0$,
$$
u\in W_2^1(0,t;L_2(D))\cap L_\infty(0,t;\mathring{H}^1(D))\cap L_1((0,t), D(\Delta_1))\,,\quad u(t)\in\mathcal{A}\,,
$$
which satisfies the energy estimate
\begin{equation}
\frac{1}{2} \int_0^t \|\partial_t u(s)\|_2^2\ \mathrm{d}s + \mathcal{E}_W(u(t)) \le \mathcal{E}_W(u_0) \label{EvEq010}
\end{equation}
and the weak formulation of \eqref{EvEq001ax}
\begin{equation}
\int_D (u(t)-u_0) \vartheta\ \mathrm{d}x = - \int_0^t \int_D \left[ \nabla u \cdot \nabla\vartheta + \zeta_u \vartheta +\lambda\vartheta g_W(u) \right]\ \mathrm{d}x\mathrm{d}s \label{EvEq011}
\end{equation}
for all $\vartheta\in \mathring{H}^1(D)\cap L_\infty(D)$, where
$$
\zeta_u:= \Delta u-\lambda g_W(u)-\partial_t u\in  L_1((0,t)\times D)
$$
satisfies $\zeta_u\in \partial\mathbb{I}_{[-1,\infty)}(u)$  a.e. in $(0,t)\times D$.
\end{definition}

The existence of  energy solutions is guaranteed by the next theorem.

\begin{thm}[{\bf Existence}]\label{EvThm002}
Suppose \eqref{W1} and let $\lambda>0$. Given $u_0\in \mathcal{A}\cap L_\infty(D)$, there exists at least one energy solution $u$
 to \eqref{EvEq001x}  satisfying also
\begin{equation}
u(t,x) \le \| (u_0)_+\|_\infty\ , \qquad (t,x)\in (0,\infty)\times D\ . \label{EvEq003}
\end{equation}
In addition, if there are $\kappa\in (0,1)$ and $T>0$ such that $u\ge \kappa-1$ in $(0,T)\times D$, then 
$$
u\in C^1([0,T);L_p(D))\cap C((0,T);W_p^2(D))
$$
for all $p\in (1,\infty)$.
\end{thm}

The proof of Theorem~\ref{EvThm002} is performed in Section~\ref{Sec5} and relies partially on the gradient flow structure in $L_2(D)$ of \eqref{EvEq001x}. Indeed, we exploit this structure under the additional assumption  $1/W\in L_4(D)$. In that case, we use the direct method of calculus of variations to construct a solution in $\mathcal{A}\cap H^2(D)$ to the time implicit Euler scheme associated with \eqref{EvEq001x}. We then use a compactness argument along with \cite{BS73} to solve the same implicit Euler scheme but with $1/W\in L_2(D)$, thereby obtaining a less regular solution in $\mathcal{A}\cap D(\Delta_1)$. We next pass to the limit as the discretization parameter tends to zero, using a combination of energy arguments and Dunford-Pettis' theorem.

We supplement Theorem~\ref{EvThm002} with a uniqueness result which is valid when $1/W$ enjoys better integrability property. 

\begin{thm}[{\bf Uniqueness and Comparison Principle}] \label{EvThm005}
Suppose \eqref{W1} and, in addition, that 
$1/W\in L_r(D)$ for some $r>3d/2$. Let $\lambda>0$. Given $u_0\in \mathcal{A}\cap L_\infty(D)$, there exists a unique energy solution $u$  
to \eqref{EvEq001x}  satisfying also~\eqref{EvEq003}. 

Furthermore, if $v_0\in \mathcal{A}\cap L_\infty(D)$ is such that $u_0\le v_0$ in $D$ and if $v$ 
denotes the corresponding energy solution to \eqref{EvEq001x}, then $u(t)\le v(t)$ in $D$ for all $t\ge 0$.
\end{thm}

It is well-known that the comparison principle is available for parabolic variational inequalities \cite[Proposition~II.7]{Br72}. The proof of Theorem~\ref{EvThm005} is given in Section~\ref{Sec5}.

We next turn to the large time dynamics of energy solutions to \eqref{EvEq001x} and combine the information on the maximal stationary solutions provided by Theorem~\ref{statsol} along with the energy inequality \eqref{EvEq010} to describe the structure of the $\omega$-limit set $\omega(u_0)$, defined for an energy solution $u$ to \eqref{EvEq001x} as the set of all $v\in \mathcal{A}$ for which there is a sequence $(t_k)_{k\ge 1}$ of positive real numbers such that
\begin{equation*}
\lim_{k\to\infty} t_k = \infty \;\text{ and }\; \lim_{k\to\infty} \|u(t_k)-v\|_2 = 0\ . 
\end{equation*} 
Owing to the energy structure, the $\omega$-limit set consists only of stationary solutions.

\begin{thm}\label{EvThm004}
Suppose \eqref{W1}. Let $\lambda>0$ and $u_0\in \mathcal{A}\cap L_\infty(D)$ and consider an energy solution $u$ to \eqref{EvEq001x}  satisfying also \eqref{EvEq003}. 
Then the set $\omega(u_0)$ is non-empty and bounded in $\mathring{H}^1(D)$ and contains only stationary solutions to \eqref{EvEq001x}. Furthermore, if $1/W\in L_r(D)$ for some $r>3d/2$ and $u_0\ge U_\lambda$ in $D$, then $\omega(u_0)=\{U_\lambda\}$ and 
$$
\lim_{t\to\infty} \| u(t) - U_\lambda\|_2 = 0\ . 
$$
\end{thm}

The proof of Theorem~\ref{EvThm004} relies on the energy inequality \eqref{EvEq010} and is carried out in Section~\ref{Sec6} in the spirit of the proof of Lasalle's invariance principle. An numerical illustration is given in Figure~\ref{MassSpring}.

\begin{figure}[ht]
\centering\includegraphics[scale=.45]{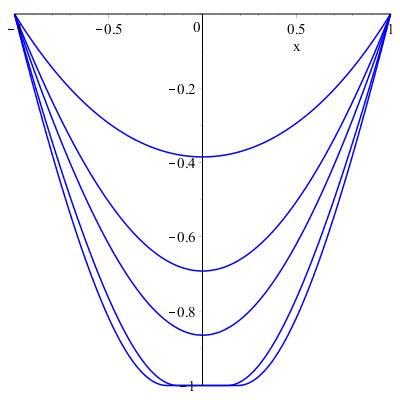}
\caption{\small Zipped and unzipped states: one-dimensional simulation of the solution to \eqref{EvEq001x} at increasing time instants ($W\equiv 1$, $\lambda=4$, $u_0= 0$ and constraint approximated by $K\min(1+u,0)$ with $K$ large).}\label{MassSpring}
\end{figure}

We finally provide some additional information on the dynamics when the evolution starts from rest, that is, when $u_0=0$.

\begin{thm}\label{obelix888}
Suppose that $1/W\in L_r(D)$ for some $r>3d/2$ and let $u$ be the solution to \eqref{EvEq001x} with $u_0=0$. Then:
\begin{itemize}
\item[(i)] For $t_1<t_2$ there holds $u(t_1)\ge u(t_2)$ in $D$ and $\mathcal{C}(u(t_1))\subset \mathcal{C}(u(t_2))$.

\item[(ii)] If $\lambda<\Lambda_z$, then $u(t)$ is unzipped for all $t\ge 0$.

\item[(iii)] There is $\Lambda^*\ge \Lambda_z$ such that if $\lambda >\Lambda^*$, then there is $T_z=T_z(\lambda, W)>0$ such that $u(t)$ is zipped for $t>T_z$.

\item[(iv)] In addition, $\Lambda^*= \Lambda_z$ if $W\in L_\infty(D)$.
\end{itemize}
\end{thm}

For the proof of Theorem~\ref{obelix} is performed in Section~\ref{Sec6}. The time monotinicity in statement (i) is actually a classical feature of parabolic equations when the initial value is a supersolution and it turns out that the constraint does not alter this property.

\section{Stationary Solutions}\label{Sec2}

In this section we prove Theorem~\ref{statsol}. Thus we investigate stationary solutions to \eqref{EvEq001x} in the sense of Definition~\ref{SSDef001}, that is, solutions to
\begin{subequations}\label{idefix} 
\begin{align}
- \Delta u +\partial \mathbb{I}_{[-1,\infty)}(u)&\owns   -\lambda g_W(u) \,,&& x\in D\,, \label{idefix1}\\
u&=0\,,&& x\in\partial D\,, \label{idefix3}
\end{align}
\end{subequations}
with $\lambda>0$ and $g_W$ given in \eqref{g}. Recalling that we always assume \eqref{W1} to hold, the right-hand side of \eqref{idefix1} belongs to $L_1(D)$ and the analysis of this section is based on the nice properties of the maximal monotone operator $-\Delta+\partial\mathbb{I}_{[-1,\infty)}$ in $L_1(D)$ thoroughly studied in \cite[$\S$1]{BS73}. In particular, we recall the basic result on existence and uniqueness.

\begin{thm}\cite[Theorem 1]{BS73}\label{Flea}
Given $f\in L_1(D)$, there is a unique $v\in D(\Delta_1)$ such that
$$
- \Delta v (x) +\partial \mathbb{I}_{[-1,\infty)}(v(x))\owns   f(x)\,,\quad  x\in D\,. 
$$
\end{thm}

We now establish the existence of $L_1$-solutions to \eqref{idefix} with the help of a classical monotone scheme. To this end, we introduce the notion of {\it subsolution} and  {\it supersolution} to \eqref{idefix}.

\begin{definition}\label{DD}
(a) A subsolution to \eqref{idefix} is a function $\sigma\in D(\Delta_1)$  with $g_W(\sigma)\in L_1(D)$ for which there is $F_\sigma\in L_1(D)$ such that $F_\sigma\le -\lambda g_W(\sigma)$ in $D$, and $\sigma$ is the unique solution to
\begin{equation*}
-\Delta \sigma +\partial\mathbb{I}_{[-1,\infty)}(\sigma)\owns F_\sigma\ \text{ in } D\,.
\end{equation*}

(b) A supersolution to \eqref{idefix} is a function $\sigma\in D(\Delta_1)$  with $g_W(\sigma)\in L_1(D)$ for which there is $F_\sigma\in L_1(D)$ such that $F_\sigma\ge -\lambda g_W(\sigma)$ in $D$, and $\sigma$ is the unique solution to
\begin{equation*}
-\Delta \sigma +\partial\mathbb{I}_{[-1,\infty)}(\sigma)\owns F_\sigma\ \text{ in } D\,.
\end{equation*}
\end{definition}

We first observe that, for any subsolution $\sigma$ to \eqref{idefix}, we have 
\begin{equation}\label{b}
-1\le \sigma\le 0\quad\text{in}\    D\,,
\end{equation}
 where the first inequality stems from Definition~\ref{DD}~(a) and the second one is due to \cite[Proposition 5]{BS73} by comparison with the zero solution since $F_\sigma\le -\lambda g_W(\sigma)\le 0$ in $D$.

\begin{prop}[{\bf Stationary Solutions}]
\label{P1}
Let $\lambda>0$. Then there is a solution $U_\lambda\in D(\Delta_1)\cap \mathring{H}^1(D)$ to \eqref{idefix} with $-1\le U_\lambda\le 0$ a.e.
Moreover, this solution is maximal in the sense that $U_\lambda\ge \sigma$ in $D$ for any subsolution $\sigma$ in the sense of Definition~\ref{DD}.
Finally, if $1/W\in L_{2p}(D)$ for some $p\in (1,\infty)$, then $U_\lambda\in W_p^2(D)$.
\end{prop}

\begin{proof}
Let us first observe that there is at least one subsolution to \eqref{idefix}. Indeed, since $1/W\in L_1(D)$, it follows from Theorem~\ref{Flea} that there exists a unique solution $\sigma_0\in D(\Delta_1)$ to
$$
-\Delta \sigma_0 +\partial\mathbb{I}_{[-1,\infty)}(\sigma_0) \owns -\frac{\lambda}{2W^2}\quad \text{in } \ D\,.
$$
As $\sigma_0\ge -1$ in $D$, one has that $-\lambda/2W^2\le -\lambda g_W(\sigma_0)$ in $D$, so that $\sigma_0$ is a subsolution to \eqref{idefix} in the sense of Definition~\ref{DD}~(a).

Fix now an arbitrary subsolution $\sigma$ with corresponding  $F_\sigma$  and set $u^0:=0$ in $D$. Since $\sigma\le u^0$ in $D$ by \eqref{b}, we have 
$$
F_\sigma\le -\lambda g_W(\sigma)\le -\lambda g_W(u^0)\le 0\quad \text{in }\ D\,.
$$
Hence, if $u^1\in D(\Delta_1)$ denotes the unique solution to
$$
-\Delta u^1 +\partial\mathbb{I}_{[-1,\infty)}(u^1) \owns -\lambda g_W(u^0)\quad\text{in } \ D
$$
given by Theorem~\ref{Flea}, then \cite[Proposition 5]{BS73} implies that $-1\le \sigma\le u^1\le u^0=0$ in $D$ and
$$
\zeta_\sigma:=F_\sigma+\Delta \sigma\le \zeta^1:=-\lambda g_W(u^0)+\Delta u^1\le \zeta^0:=0\quad \text{in } \ D\,.
$$
Arguing by induction yields for each $n\in \N$ the unique solution  $u^{n+1}\in D(\Delta_1)$  to
$$
-\Delta u^{n+1} +\partial\mathbb{I}_{[-1,\infty)}(u^{n+1}) \owns -\lambda g_W(u^{n})\quad \text{in }\ D
$$ 
for which
\begin{align}
-1&\le \sigma\le u^{n+1}\le u^n\le 0 \quad \text{in }\ D \,,\label{est1}\\
\zeta_\sigma &\le \zeta^{n+1}\le \zeta^n\le 0\quad \text{in }\ D \,,\label{est2}
\end{align}
where
\begin{equation}\label{hakim}
\zeta^n:=-\lambda g_W(u^{n-1})+\Delta u^{n}\,,\quad n\ge 1\,.
\end{equation}
Since $\sigma$ and $\zeta_\sigma$ both belong to $L_1(D)$, the ordering properties \eqref{est1} and \eqref{est2} allow us to apply the monotone convergence theorem and obtain that
\begin{equation}\label{sting}
(u^n,\zeta^n)\longrightarrow (U_\lambda,\zeta)\quad \text{in }\ L_1(D,\R^2)\,,
\end{equation}
where, for $x\in D$,
$$
U_\lambda (x):=\inf_{n\ge 0} u^n(x)\qquad\text{and}\qquad \zeta (x):=\inf_{n\ge 0} \zeta^n(x)\,.
$$
Furthermore, by \eqref{est1},
$$
0\le g_W(u^n)\le g_W(u^{n+1})\le \frac{1}{2W^2}\quad \text{in } \ D\,,
$$
and we use once more the monotone convergence theorem to deduce that there is $G\in L_1(D)$ such that
\begin{equation}\label{copeland}
g_W(u^n)\longrightarrow G\quad \text{in } L_1(D)\,.
\end{equation}
Since there is a subsequence $(n_k)_{k\in\N}$ such that 
$$
\big(u^{n_k},g_W(u^{n_k})\big)\longrightarrow (U_\lambda, G)\quad \text{a.e. in }\ D
$$
according to \eqref{sting} and \eqref{copeland}, the continuity of $g_W$ (with respect to $u$) entails that
\begin{equation}\label{summers}
G=g_W(U_\lambda)\,.
\end{equation}
On the one hand, we pass to the limit as $n\rightarrow\infty$ in \eqref{hakim} by using \eqref{sting}-\eqref{summers} to obtain
\begin{equation}\label{incognito}
-\Delta U_\lambda +\zeta=-\lambda g_W(U_\lambda)\quad \text{in } \ D\,.
\end{equation}
On the other hand, let $v\in D(\Delta_1)$ be the unique solution to
$$
-\Delta v+ \partial\mathbb{I}_{[-1,\infty)}(v) \owns -\lambda g_W(U_\lambda)\quad\text{in }\  D\,.
$$
Due to \cite[Proposition 5]{BS73}, we have
$$
\|\Delta u^{n+1}-\Delta u^n\|_1\le 2\lambda \|g_W(u^n)-g_W(U_\lambda)\|_1\,.
$$
Thanks to \eqref{sting}-\eqref{summers}, we may pass to the limit as $n\rightarrow\infty$ in the previous inequality and conclude that
$\Delta U_\lambda = \Delta v\in L_1(D)$. This implies $U_\lambda =v\in D(\Delta_1)$ and we derive from \eqref{incognito} that $\zeta\in \partial\mathbb{I}_{[-1,\infty)}(U_\lambda)$. Consequently, $U_\lambda$ is a solution to \eqref{idefix}. Moreover, $U_\lambda$ is independent of the previously fixed subsolution $\sigma$ (since the sequence $(u^n)$ is) and hence $U_\lambda$ lies above any subsolution.

Finally, if $1/W\in L_{2p}(D)$ for some $p\in (1,\infty)$, then 
$-\lambda g_W(U_\lambda)\in L_p(D)$ so that \cite[Theorem 1, Corollary 8]{BS73} readily imply that $U_\lambda\in W^2_p(D)$.
This completes the proof.
\end{proof}

We now draw several consequences from Proposition~\ref{P1} and begin with the monotonicity of $U_\lambda$ with respect to $\lambda$.

\begin{cor}\label{C1}
If $\lambda_1<\lambda_2$, then $U_{\lambda_1}\ge U_{\lambda_2}$. In particular, if $U_\lambda$ is a zipped state, then $U_{\lambda'}$ is also zipped for any $\lambda'>\lambda$.
\end{cor}

\begin{proof}
This follows from Proposition~\ref{P1} by observing that
$$
-\Delta U_{\lambda_2}+\zeta_{\lambda_2} =-\lambda_2 g_W(U_{\lambda_2})\le -\lambda_1 g_W(U_{\lambda_2})\,.
$$
\end{proof}

A monotonicity property with respect to $W$ is also available.

\begin{cor}\label{C1a}
For $j\in\{1,2\}$ let $W_j$ be a non-negative measurable function such that  $1/W_j\in L_2(D)$ with $W_1\le W_2$ in $D$ and let $\lambda>0$. If $U_{\lambda,j}$ denotes the maximal solution to \eqref{idefix} corresponding to $W_j$, then $U_{\lambda,1}\le U_{\lambda,2}$.
\end{cor}

\begin{proof}
The assumptions imply that
$$
-\lambda g_{W_1}(U_{\lambda,1})\le -\lambda g_{W_2}(U_{\lambda,1}) \quad \text{ in }\ D
$$
and the assertion thus follows from Proposition~\ref{P1}.
\end{proof}

We next turn to the structure of the set of stationary solutions and introduce
$$
\Lambda_z:=\inf\{\lambda>0\,;\, U_\lambda\ \text{is zipped}\}\in [0,\infty]\,.
$$

\begin{cor}\label{C1b}
If $\lambda>\Lambda_z$, then any solution to \eqref{idefix} is zipped.
\end{cor}

\begin{proof}
This follows immediately from Corollary~\ref{C1} and the maximality of $U_\lambda$ when $\Lambda_z$ is finite.
\end{proof}

We now investigate in more detail the touchdown behavior of  solutions. As we shall see in the next result, zipped states do exist for large values of $\lambda$.

\begin{prop}[\textbf{Zipped Solutions}]\label{P3}
The threshold value $\Lambda_z$ is  finite. 
\end{prop}

\begin{proof}
We argue along the lines of \cite{GPW05, Ka63, Pe02}.  Let $\varphi_1\in H^2(D)\cap \mathring{H}^1(D)$ be the positive eigenfunction of $-\Delta_1$ associated with the positive first eigenvalue $\mu_1$ and satisfying $\|\varphi_1\|_1=1$. Consider $\lambda>0$ and let $u$ be a solution to \eqref{idefix}. Multiplying \eqref{idefix1} by $\varphi_1$ and integrating over $D$ entail
\begin{equation*}
\begin{split}
\mu_1\int_D\varphi_1 u\,\rd x +\int_D \zeta \varphi_1\,\rd x =-\lambda\int_D \varphi_1 g_W(u)\,\rd x\,.
\end{split}
\end{equation*}
Owing to \eqref{b}, we have
\begin{equation*}
\begin{split}
\lambda\int_D \varphi_1 g_W(u)\,\rd x \ge \lambda\int_D \varphi_1 g_W(0)\,\rd x
\end{split}
\end{equation*}
and 
\begin{equation*}
\begin{split}
\int_D\varphi_1 u\,\rd x \ge -\int_D \varphi_1 \,\rd x =-1\,.
\end{split}
\end{equation*}
Therefore
$$
\int_D \zeta \varphi_1\,\rd x \le - \lambda\int_D\varphi_1 g_W(0)\,\rd x +\mu_1 < 0\,,
$$
as soon as
$$
\lambda > \lambda^*:=\frac{\mu_1}{\|\varphi_1 g_W(0)\|_1}\,.
$$
Since $\varphi_1>0$ in $D$, we have thus shown that $\zeta\not\equiv 0$, so that $u$ is a zipped state. In particular, $U_\lambda$ is zipped for $\lambda>\lambda^*$, hence $\Lambda_z\le \lambda^*$.
\end{proof}

 Now we show that the maximal solution is unzipped for small values of $\lambda$.

\begin{prop}[{\bf Unzipped Solutions}]
\label{P2}
The threshold value $\Lambda_z$ is  positive. Furthermore, if $1/W\in L_{2p}(D)$ with $p>d/2$, there is $\lambda_*:=\lambda_*(W)>0$ such that for $\lambda\in (0,\lambda_*)$ there is $\omega_\lambda>0$ such that \emph{any} solution $u$ to \eqref{idefix} satisfies 
$$
u\ge -1 +\omega_\lambda \ \text{ in $D$}.
$$ 
In particular, all states are unzipped for $\lambda\in (0,\lambda_*)$.
\end{prop}

\begin{proof}
Firstly, it follows from \cite{Pe02,GhG06}  that there is $\lambda_0>0$ such that the boundary value problem
$$
-\Delta V_\lambda=-\lambda g_0(V_\lambda)\quad\text{in $D$}\,,\qquad V_\lambda=0\quad\text{on $\partial D$}
$$
has a solution $V_\lambda\in H^2(D)\cap \mathring{H}^1(D)$ satisfying
 $V_\lambda>-1$ in $D$ for all $\lambda\in (0,\lambda_0)$. Since $g_0(V_\lambda)\ge g_W(V_\lambda)$ in $D$, we deduce that
$V_\lambda$ is a subsolution to \eqref{idefix} so that Proposition~\ref{P1} implies that $U_\lambda\ge V_\lambda$  in $D$. Therefore, $\Lambda_z\ge \lambda_0$.

Secondly, if $1/W\in L_{2p}(D)$ with $p>d/2$, then there exists a unique solution $v_\lambda\in W_p^2(D)\hookrightarrow C(\bar D)$ to 
$$
- \Delta v_\lambda =   -\frac{\lambda}{2W^2} \quad\text{in $D$}\,,\qquad
v_\lambda =0\quad\text{on $\partial D$}
$$
satisfying
$$
v_\lambda = \lambda v_1 \ge -\lambda \|v_{1}\|_\infty>-1
$$ 
in $D$ provided that $\lambda\in (0,\lambda_*)$ for some $\lambda_*$ sufficiently small. It then remains to note that
$$
-g_W(u)\ge -\frac{1}{2W^2} \quad\text{a.e. in $D$}
$$
for any solution $u$ to \eqref{idefix} so that $u\ge v_\lambda$ a.e. in $D$ according to \cite[Proposition~5]{BS73}.
\end{proof}

Note that the maximal solution $U_\lambda$ is unzipped for $\lambda\in (0,\Lambda_z)$ while it is zipped for $\lambda\in (\Lambda_z,\infty)$ according to Proposition~\ref{P3} and Proposition~\ref{P2}. Consequently, Theorem~\ref{statsol} is now a consequence of the preceding observations.

\medskip

A key issue is whether or not unzipped and zipped states may coexist for a given value of $\lambda$. Introducing
$$
\Lambda_u:=\inf\{\lambda>0\,;\, \text{there is a zipped state to \eqref{idefix}}\}\in [0,\Lambda_z]\,,
$$
it follows from Proposition~\ref{P2} that $\Lambda_u>0$ at least if $1/W\in L_{2p}(D)$ for $p>d/2$. Coexistence of zipped and unzipped states could only take place in the intermediate range $ [\Lambda_u, \Lambda_z]$ provided this interval is non-empty. This issue  will be addressed in the next section when $d=1$ and $W\equiv const$. In the related model studied in \cite{GB01} this phenomenon seems indeed to occur according to the simulations performed therein.

\section{A Complete Characterization of Stationary Solutions when $d=1$ and $W=const$}\label{Sec3}

We now derive a complete characterization of the solutions to \eqref{idefix} in dimension $d=1$ when \mbox{$W\equiv const >0$}. Without loss of generality we let $D=(-1,1)$. We first state a simple characterization of zipped states.

\begin{lem}\label{L11}
If $u$ is a zipped state to \eqref{idefix} with $D=(-1,1)$, then there are $0<b<a<1$ such that $\mathcal{C}(u)=[b,a]$, $u'(b)=u'(a)=0$, and $u(x)>-1$ for $x\in (-1,b)\cup (a,1)$.
\end{lem}

\begin{proof}
According to \cite[II.~Theorem~7.1]{KS}, any solution $u$ to \eqref{idefix} belongs to $C^1([-1,1])$. Let $u$ be a zipped state. Owing to $u(-1)=0$ there is a $b\in (0,1)$ such that $u(x)>-1$ for $x\in [-1,b)$ and $u(b)=-1$. Moreover, $u'(b)=0$ since $u'$ is continuous and $b$ is a minimum point of $u$. Assume now for contradiction that there are $b_n\searrow b$ such that $u(b_n)>-1$. Then $u''(b_n)= \lambda g_W(u)(b_n)>0$ so that $u'(b_n)>0$. Hence
$$
u'(x)=u'(b_n)+\int_{b_n}^x \lambda g_W(u)(y)\,\rd y \ge u'(b_n) >0\,,\quad x \in [b_n,1)\,,
$$
implying that $b$ is the only point at which $u$ takes the value $-1$, contradicting that $u$ is a zipped state. Consequently, $u(x)=-1$ for $x$ in a right-neighborhood of $b$. Setting
$$
a:=\sup\{x\in (b,1)\,;\, u(y)=-1\text{ for } y\in [b,x]\} <1\,,
$$
there exist $a_n\searrow  a$ such that $u(a_n)>-1$. The same argument as above shows that $u'(x)>0$ for $x\in (a_n,1)$. This implies the statement.
\end{proof}

\begin{rem}
Lemma~\ref{L11} holds true for any continuous non-negative function $W$.
\end{rem}

Introducing
\begin{equation*}
\varphi(r):=\sqrt{r(1-r)}+(1-r)^{3/2}\log(1+\sqrt{r})-\frac{1}{2} (1-r)^{3/2}\log (1-r)\,,\qquad r\in (0,1)\,,
\end{equation*}
and
\begin{equation*}
\Lambda_*(r):= (1+r)^3 \varphi\left(\frac{1}{1+r}\right)^2\,,\quad r\ge 0\,,
\end{equation*}
we can give a complete characterization of the solutions to \eqref{idefix} when $d=1$ and $W=const>0$ in form of a case-by-case analysis.

\begin{thm}\label{T44}
 Let $d=1$ and $W=const>0$. There is a unique $r_0\in (0,1)$ such that $\varphi'(r_0)=0$, and the solutions to \eqref{idefix} are characterized as follows. 

{\bf (I)} For $1/(1+W) \le r_0$, the following possibilities arise:
\begin{itemize}

\item[(i)] If $\lambda>\Lambda_*(W)$, then there is no unzipped state and a unique zipped state.

\item[(ii)] If $\lambda=\Lambda_*(W)$, then there is a unique unzipped state that touches down on $-1$ at exactly one point, but no zipped state.

\item[(iii)] If $\lambda<\Lambda_*(W)$, then there is a unique unzipped state,  but no zipped state.
\end{itemize}

{\bf (II)} For $1/(1+W) > r_0$, the following possibilities arise:
\begin{itemize}

\item[(i)] If $\lambda>(1+W)^3\varphi(r_0)^2$, then there is no unzipped state, but a unique zipped state.

\item[(ii)] If $\lambda=(1+W)^3\varphi(r_0)^2$, then there are a unique unzipped and a unique zipped state.

\item[(iii)] If $\lambda\in \left(\Lambda_*(W),(1+W)^3\varphi(r_0)^2\right)$, then there are two unzipped states and a unique zipped state.

\item[(iv)] If $\lambda=\Lambda_*(W)$, then there are two unzipped states, one touching down exactly at one point, but no zipped state.

\item[(v)] If $\lambda<\Lambda_*(W)$, then there are two unzipped states, but no zipped state.
\end{itemize}
\end{thm}

The value of $r_0$ is approximately $0.388346$. Theorem~\ref{T44} shows that the structure of stationary solutions is very sensitive with respect to the value of $\lambda$. In case (I) (corresponding to large values of $W$), there is no coexistence of zipped and unzipped states. However, in case (II) (corresponding to small values of $W$) there is an interval for $\lambda$ for which there is coexistence. This is in accordance with the numerical findings of \cite{GB01} for a related problem (see Figure~3 therein). That the structure of stationary solutions is very sensitive with respect to the value of $\lambda$ has also been observed in related MEMS models without constraint but including a quasilinear diffusion given by the mean curvature, see \cite{BP12,PX15,CHW13}.

\medskip

To prove Theorem~\ref{T44} we first characterize  the zipped states of \eqref{idefix}. To this end we investigate the shooting problem of finding $a\in (-1,1)$ such that there is a solution $u$ to
\begin{subequations}\label{idefixb} 
\begin{align}
- u''&=   -\lambda g_W(u) \,,\qquad x\in (a,1)\,, \label{idefix1b}\\
u(a)=-1&\,,\qquad u'(a)=0\,,\qquad u(1)=0\,,&&  \label{idefix3b}
\end{align}
\end{subequations}
where $g_W$ is given by \eqref{g}. The next lemma discusses its solvability completely.

\begin{lem}\label{31}
The shooting problem~\eqref{idefixb} with $a\in (-1,1)$ has a solution if and only if the constraint
\begin{equation}\label{c4}
0<\frac{\Lambda_*(W)}{4}<\lambda
\end{equation}
is satisfied. In that case, $a$ is uniquely given by
\begin{equation}\label{c44}
a=1-\sqrt{\frac{\Lambda_*(W)}{\lambda}}\,.
\end{equation}
\end{lem}

\begin{proof}
Consider $a\in (-1,1)$ such that \eqref{idefixb} has a solution $u$ on $[a,1]$. Multiplying \eqref{idefix1b} by $2 u'$ yields
$$
\frac{\rd}{\rd x} (u')^2=\frac{\lambda u'}{(1+u+W)^2}=-\lambda \frac{\rd}{\rd x} \left(\frac{1}{1+u+W}\right) \,.
$$
Integrating this equality and using $u'\ge 0$ due to the convexity of $u$ we derive
$$
\left(\frac{1+u+W}{1+u}\right)^{1/2} u'=\left(\frac{\lambda}{W}\right)^{1/2}\,.
$$
Integrating then this relation from $a$ to $1$ implies the constraint
\begin{equation}\label{c}
\left(\frac{W}{\lambda}\right)^{1/2}\int_{-1}^0 \left(\frac{1+z+W}{1+z}\right)^{1/2}\,\rd z=1-a \in (0,2)\,.
\end{equation}
Observe that the substitution $y=\sqrt{1+z}$ gives 
\begin{equation*}
\begin{split}
\int_{-1}^0 \left(\frac{1+z+W}{1+z}\right)^{1/2}\,\rd z&=2\int_0^1 \left(y^2+W\right)^{1/2}\,\rd y\\
&= \left( y\sqrt{y^2+W}+W\log\left(y+\sqrt{y^2+W}\right)\right)\Big\vert_{y=0}^{y=1}
\\
&=\sqrt{1+W}+W\log\left(1+\sqrt{1+W}\right)-\frac{1}{2} W\log W\,.
\end{split}
\end{equation*}
Combining this identity with \eqref{c}, we deduce that the shooting problem \eqref{idefixb} has a solution provided that
\begin{equation*}
\lambda^{1/2} (1-a)= \sqrt{W(1+W)}+W^{3/2}\log\left(1+\sqrt{1+W}\right)-\frac{1}{2} W^{3/2}\log W\\
=\sqrt{\Lambda_*(W)}\,.
\end{equation*}
This shows that \eqref{c4} is a necessary condition for the solvability of \eqref{idefixb} and that $a$ is given  by \eqref{c44}. If \eqref{c4} is satisfied, then we may define $a\in (-1,1)$ by \eqref{c44} and thereby obtain a solution to \eqref{idefixb}.
\end{proof}

\begin{cor}\label{C100}
Any solution to \eqref{idefix} is even on $(-1,1)$. Moreover,  \eqref{idefix} admits a zipped state if and only if the constraint
\begin{equation}\label{c4x}
0<\Lambda_*(W)<\lambda
\end{equation}
is satisfied. In that case, the zipped state is unique and its coincidence set is $[-a,a]$ with $a$ given by \eqref{c44}.
\end{cor}

\begin{proof}
If $u$ is any unzipped solution to \eqref{idefix}, then $u$ is even: Indeed, if $r_0\in (-1,1)$ is the (unique) point of minimum of the strictly convex function $u$, then $x\mapsto u(r_0- x)$ and $x\mapsto u(r_0+x)$ coincide as they solve the same ordinary differential equation with identical initial values $(u(r_0),0)$ at $x=0$. Since $u(-1)=u(1)=0$, this readily implies $r_0=0$ and that $u$ is even. If $u$ is any zipped solution to \eqref{idefix}, then $u=-1$ on the coincidence set $[b,a]$ according to Lemma~\ref{L11}. Hence, $u$ solves \eqref{idefixb} on $[a,1]$ while $x\mapsto u(-x)$ solves \eqref{idefixb} on $[-b,1]$. By Lemma~\ref{31}  the constraint \eqref{c4} is satisfied, and it follows from \eqref{c44} that $a=-b$. Therefore, since \eqref{idefixb} has a unique solution, this implies that $u(x)=u(-x)$ for $x\in [a,1]$. As $u=-1$ on $[-a,a]$, this shows that $u$ is even. Finally, since $a\in (0,1)$ is uniquely determined by \eqref{c44}, the assertion follows from Lemma~\ref{31}.
\end{proof}

Note that by Corollary~\ref{C100}, any unzipped state of \eqref{idefix}  reaches its minimum value at $x=0$. Therefore, to characterize all unzipped states of \eqref{idefix} it suffices to investigate the following shooting problem of finding $m\in (0,1]$ such that there is a solution to
\begin{subequations}\label{obelix}
\begin{align}
- u''&=   -\lambda g_W(u) \,,\qquad x\in (0,1)\,, \label{obelix1}\\ 
 u(0)=-m&\,,\qquad u'(0)=0\,,\qquad u(1)=0\,. && \label{obelix2}  
\end{align}
\end{subequations}
 
As for its solvability we have:

\begin{lem}\label{miraculix}
The shooting problem \eqref{obelix} with $m\in (0,1]$ has a solution if and only if the constraint
\begin{equation}\label{c5a}
\frac{\sqrt{\lambda}}{(1+W)^{3/2}}\in \varphi\left(\left(0,\frac{1}{1+W}\right]\right)
\end{equation}
is met. In that case, $m$ satisfies
\begin{equation}\label{c5}
\lambda= (1+W)^3\varphi\left(\frac{m}{1+W}\right)^2\,.
\end{equation}
\end{lem}

\begin{proof}
Let $m\in (0,1)$ be such that \eqref{obelix} has a solution $u$. Proceeding as in the proof of Lemma~\ref{31}, the relation corresponding to \eqref{c} reads
\begin{equation*}
\int_{-m}^0 \left(\frac{1+z+W}{m+z}\right)^{1/2}\,\rd z=\left(\frac{\lambda}{1-m+W}\right)^{1/2}\,.
\end{equation*}
Computing then the left-hand side with the substitution $y=\sqrt{m+W}$, we obtain
\begin{equation*}
\begin{split}
\sqrt{\lambda}= & \left[m(1+W)(1+W-m)\right]^{1/2}+(1+W-m)^{3/2}\log\left(\sqrt{m}+\sqrt{1+W}\right)\\
&-\frac{1}{2} (1+W-m)^{3/2}\log(1+W-m)\,,
\end{split}
\end{equation*}
which is equivalent to \eqref{c5}. Conversely, if \eqref{c5a} is met, then we may choose $m\in (0,1)$ such that \eqref{c5} holds and the assertion follows.
\end{proof}

To analyze \eqref{c4x}, \eqref{c5a} we derive more information on the function $\varphi$. 

\begin{lem}\label{L99}
The function $\varphi$ is positive on $(0,1)$ with $\varphi(0)=0=\varphi(1)$, and there is a unique $r_0\in (0,1)$ such that $\varphi'(r_0)=0$.
\end{lem}

\begin{proof}
For $r\in (0,1)$ the derivative of $\varphi$ is of the form
$$
\varphi'(r)=(1-r)^{1/2}\,\psi(r)\,,
$$
where
$$
\psi(r):=\frac{1-2r}{2\sqrt{r}(1-r)}-\frac{3}{2}\log(1+\sqrt{r})+\frac{1-r}{2(\sqrt{r}+r)}+\frac{3}{4}\log(1-r)+\frac{1}{2}\,.
$$
For the derivative of $\psi$ we obtain
$$
\psi'(r)=\frac{-2r^2-1+r}{4r^{3/2} (1-r)^2}-\frac{5}{4(r+\sqrt{r})}-\frac{(1-r)\left(1+2\sqrt{r}\right)}{4\sqrt{r}(\sqrt{r}+r)^2}-\frac{3}{4 (1-r)}\,.
$$
Noticing that $-2r^2-1+r<0$ we conclude that $\psi'(r)<0$ for each $r\in (0,1)$. Thus, since $\psi(0)=+\infty$ and $\psi(1)=-\infty$, $\psi$ has a unique zero in $(0,1)$. This implies the claim.
\end{proof}

\begin{proof}[Proof of Theorem~\ref{T44}]
We discuss the different cases listed in the statement and use, to this end, the properties of $\varphi$ derived in Lemma~\ref{L99}.

\medskip

\noindent {\bf (I)} Suppose that $1/(1+W) \le r_0$.

(i) If $\lambda>\Lambda_*(W)$, then it readily follows from Corollary~\ref{C100} and Lemma~\ref{miraculix} that
there is a unique zipped state but no unzipped state.

(ii) If $\lambda=\Lambda_*(W)$, then Lemma~\ref{miraculix} implies that there is a one unzipped state touching down on~$-1$ exactly at $x=0$, while there is no zipped state according to Corollary~\ref{C100}.

(iii) If $\lambda<\Lambda_*(W)$, then there is a unique $m\in (0,1)$ such that \eqref{c5} holds. Hence, there is a unique unzipped state due to Lemma~\ref{miraculix} but no zipped state due to Corollary~\ref{C100}. 

\medskip

\noindent {\bf (II)} Now suppose that $1/(1+W) > r_0$.

(i) If $\lambda>(1+W)^3\varphi(r_0)^2$, then in particular $\lambda>\Lambda_*(W)$ so that there is a unique zipped state by Corollary~\ref{C100} but no unzipped state due to Lemma~\ref{miraculix}.

(ii) If $\lambda=(1+W)^3\varphi(r_0)^2$, then in particular $\lambda>\Lambda_*(W)$ so that there is a unique zipped state due to Corollary~\ref{C100}. Moreover, there is exactly one $m\in (0,1)$, given by $m=r_0(1+W)$ such that \eqref{c5} holds true and so there is a unique unzipped state according to Lemma~\ref{miraculix}.

(iii) If $\lambda\in \left(\Lambda_*(W),(1+W)^3\varphi(r_0)^2\right)$, then there are  $0<m_1< m_2<1$ such that \eqref{c5} is satisfied by $m_1$ and $m_2$. Hence, there are two unzipped states  due to Lemma~\ref{miraculix} and a unique zipped state by Corollary~\ref{C100}.

(iv) If $\lambda=\Lambda_*(W)$, then we also find $m\in (0,1)$ such that \eqref{c5} holds true. Hence, by Lemma~\ref{miraculix} 
there are two unzipped states, one touching down on $-1$ at  exactly $x=0$. Due to Corollary~\ref{C100} there is no zipped state.

(v) If $\lambda<\Lambda_*(W)$, then there are  $0<m_1< m_2<1$ such that \eqref{c5} is satisfied by $m_1$ and $m_2$. Hence, there are two unzipped states  due to Lemma~\ref{miraculix}, but there is no zipped state according to Corollary~\ref{C100}.

\medskip

Since all cases as listed in the statement are covered, Theorem~\ref{T44} follows.
\end{proof}

\section{The evolution problem: existence and uniqueness}\label{Sec5}

We now turn to the evolution equation  \eqref{EvEq001x} which can be equivalently written in the form
\begin{subequations}\label{EvEq001}
\begin{align}
\partial_t u - \Delta u + \zeta & = - \lambda g_W(u) \;\text{ in }\; (0,\infty)\times D\ , \label{EvEq001a} \\
\zeta & \in \partial\mathbb{I}_{[-1,\infty)}(u) \;\text{ in }\; (0,\infty)\times D\ , \label{EvEq001b} \\
u & = 0 \;\text{ on }\; (0,\infty)\times \partial D\ ,\label{EvEq001c} \\
u(0) & = u_0 \;\text{ in }\; D\ , \label{EvEq001d}
\end{align}
\end{subequations}
where $\lambda>0$ is fixed.  Recall that \eqref{W1} is assumed throughout and that the energy $\mathcal{E}_W$ and the set $\mathcal{A}$ are defined in \eqref{E} and \eqref{EvEq002},  respectively. 

The first step towards the proof of Theorem~\ref{EvThm002} is the solvability of the time implicit Euler scheme associated with \eqref{EvEq001} in $\mathcal{A}\cap H^2(D)$ when $1/W \in L_4(D)$.

\begin{lem}\label{EvLem012}
Assume that $1/W\in L_4(D)$. Given $h\in (0,1)$ and $f\in \mathcal{A}$, there exists $(u,\zeta)\in \mathcal{A}\times L_2(D)$ with $u\in H^2(D)$ solving
\begin{subequations}\label{EvEq013}
\begin{align}
\frac{u-f}{h} - \Delta u + \zeta & = - \lambda g_W(u) \;\text{ in }\;  D\ , \label{EvEq013a} \\
\zeta & \in \partial\mathbb{I}_{[-1,\infty)}(u) \;\text{ in }\; D\ , \label{EvEq013b} \\
u & = 0 \;\text{ on }\; \partial D\ , \label{EvEq013c}
\end{align}
and  satisfying
\begin{equation}
\frac{1}{2h} \|u-f\|_2^2 + \mathcal{E}_W(u) \le \mathcal{E}_W(f)\ . \label{EvEq014}
\end{equation}
\end{subequations}
In addition, if $f\le M$ in $D$ for some number $M\ge 0$, then $u\le M$ in $D$.
\end{lem}

\begin{proof}
The proof relies on the direct method of calculus of variations. For $v\in\mathcal{A}$, we define 
$$
\mathcal{F}(v) := \frac{1}{2h} \|v-f\|_2^2 + \mathcal{E}_W(v)\ .
$$
Since
\begin{equation}
\mathcal{E}_W(v) \ge \frac{\|\nabla v\|_2^2}{2} - \frac{\lambda}{2} \left\| \frac{1}{W} \right\|_1\ , \qquad v\in \mathcal{A}\ , \label{EvEq015}
\end{equation}
the functional $\mathcal{F}$ is bounded from below on $\mathcal{A}$. Therefore there is a minimizing sequence $(v_j)_{j\ge 1}$ in $\mathcal{A}$ satisfying 
\begin{equation}
\mu := \inf_{v\in\mathcal{A}}\{\mathcal{F}(v)\} \le \mathcal{F}(v_j) \le \mu + \frac{1}{j}\ , \qquad j\ge 1\ . \label{EvEq016}
\end{equation}
Owing to \eqref{EvEq015} and \eqref{EvEq016}, 
$$
\|\nabla v_j\|_2^2 \le 2 \mathcal{F}(v_j) + \left\| \frac{\lambda}{W} \right\|_1 \le 2(1+\mu) + \left\| \frac{\lambda}{W} \right\|_1\ , \qquad  j\ge 1\, ,
$$
and we infer from the compactness of the embedding of $\mathring{H}^1(D)$ in $L_2(D)$ that there are $u\in \mathring{H}^1(D)$ and a subsequence of $(v_j)_{j\ge 1}$ (not relabeled) such that
\begin{align}
v_j & \rightharpoonup u \;\text{ in }\; \mathring{H}^1(D)\ , \label{EvEq017} \\
v_j & \longrightarrow u \;\text{ in }\; L_2(D) \;\text{ and a.e. in }\; D\ . \label{EvEq018}
\end{align}
It readily follows from \eqref{EvEq017} and \eqref{EvEq018} that $u\in \mathcal{A}$ while \eqref{EvEq018}, the integrability properties of $1/W$, and Lebesgue's dominated convergence theorem entail that 
$$
\lim_{j\to\infty} \int_D \frac{\mathrm{d}x}{1+v_j+W} = \int_D \frac{\mathrm{d}x}{1+v+W}\ .
$$ 
Since the convex part of $\mathcal{E}_W$ is weakly lower semicontinuous in $L_2(D)$, classical arguments imply that $u$ is a minimizer of $\mathcal{F}$ in $\mathcal{A}$. 

To derive the corresponding Euler-Lagrange equation for $u$, we pick $v\in\mathcal{A}$, $\tau\in (0,1)$, and observe that $\tau u + (1-\tau) v$ belongs to $\mathcal{A}$. The minimizing property of $u$ reads 
$$
\mathcal{F}(u) \le \mathcal{F}(\tau u +(1-\tau) v)\ ,
$$
from which we deduce that
\begin{align*}
0 & \le \frac{1}{2h} \int_D (v-u) [(1+\tau) u + (1-\tau) v -2f]\ \mathrm{d}x \\
& \quad + \frac{1}{2} \int_D \nabla(v-u)\cdot \nabla[(1+\tau) u + (1-\tau) v]\ \mathrm{d}x \\
& \quad + \frac{\lambda}{2} \int_D \frac{v-u}{(1+u+W)[1+\tau u + (1-\tau) v + W]}\ \mathrm{d}x\ .
\end{align*}
Since
$$
\left| \frac{v-u}{(1+u+W)[1+\tau u + (1-\tau) v + W]} \right| \le \frac{|v-u|}{W^2} \in L_1(D)\ ,
$$
we may pass to the limit as $\tau\to 1$ in the previous inequality and conclude that
\begin{align*}
0 & \le \frac{1}{h} \int_D (v-u)(u-f)\ \mathrm{d}x + \int_D \nabla(v-u)\cdot \nabla u\ \mathrm{d}x + \lambda\int_D (v-u) g_W(u)\ \mathrm{d}x
\end{align*}
for all $v\in\mathcal{A}$. By \cite[Proposition~2.8]{Ba10} this implies that $u\in H^2(D)$ and 
$$
- \left[ \frac{u-f}{h} + \lambda g_W(u) \right] + \Delta u \in \partial\mathbb{I}_{[-1,\infty)}(u) \;\text{ in }\; D\ .
$$
In other words, there is $\zeta\in L_2(D)$ such that $(u,\zeta)$ solves \eqref{EvEq013a}, \eqref{EvEq013b}, and \eqref{EvEq013c}. 
Next, using once more the minimizing property of $u$ entails that $\mathcal{F}(u)\le \mathcal{F}(f)$, hence \eqref{EvEq014}. 

Finally, if $f\le M$ in $D$, then 
$$
u - h \Delta u + h \zeta = f- \lambda h g_W(u) \le M \;\text{ in }\ D\ ,
$$
and the upper bound $u\le M$ readily follows from \cite[Proposition~4]{BS73} (applied with the convex function $\Phi(r) = (r-M)_+$).
\end{proof}

We next derive some monotonicity property of the iterative scheme leading eventually to the time monotonicity of the solution to the evolution equation~\eqref{EvEq001x}.

\begin{lem}\label{saturnino}
Assume that $1/W\in L_\infty(D)$ and  let $f$ be a supersolution to \eqref{idefix} in the sense of Definition~\ref{DD}. Further let $u$ be the solution to \eqref{EvEq013} constructed in Lemma~\ref{EvLem012} corresponding to $1/h\ge \lambda\|1/W\|_\infty$. Then $u\le f$ in $D$, and $u$ is a supersolution to \eqref{idefix}.
\end{lem}

\begin{proof}
Let $\zeta_f:=F_f+\Delta f$ with $F_f\ge -\lambda g_W(f)$ in $D$ according to Definition~\ref{DD}~(b). Then, by \eqref{EvEq013a}
$$
\frac{u-f}{h}-\Delta (u-f) +\zeta-\zeta_f=-\lambda g_W(u)-F_f\le \lambda \big(g_W(f)-g_W(u)\big)\quad \text{in }\ D\,.
$$
Since
$$
-\int_D (u-f) \Delta (u-f) \,\rd x\ge 0\qquad\text{and}\qquad \int_D(\zeta-\zeta_f) (u-f)\,\rd x\ge 0
$$
due to \cite[Lemma~2]{BS73} and the monotonicity of $\partial\mathbb{I}_{[-1,\infty)}$, it follows from the above inequality that
$$
\frac{1}{h}\| (u-f)_+\|_2^2\le \frac{\lambda}{2}\int_D (u-f)_+^2 \frac{1+u+W+1+f+W}{(1+u+W)^2 (1+f+W)^2}\, \rd x \le \lambda \left\|\frac{1}{W}\right\|_\infty^3 \|(u-f)_+\|_2^2\,.
$$
Owing to the assumption on $h$, this readily implies that $(u-f)_+$ vanishes identically, that is, $u\le f$ in $D$. Finally, using \eqref{EvEq013} again, we realize that
$$
-\Delta u+\zeta=-\lambda g_W(u)+\frac{f-u}{h}\ge -\lambda g_W(u)\quad \text{in }\ D\,,
$$
so that $u$ is a supersolution to \eqref{idefix}.
\end{proof}

\begin{rem}\label{waters}
The statement of Lemma~\ref{saturnino} remains true if one only assumes that $1/W\in L_r(D)$ for some $r>3d/2$ provided $h$ is sufficiently small with respect to the norm of $1/W$ in $L_r(D)$. The proof is slightly more involved as it uses the Gagliardo-Nirenberg inequality, see the proof of Theorem~\ref{EvThm005} below for a similar argument.
\end{rem}

We next establish a version of Lemma~\ref{EvLem012} under the only assumption that $1/W\in L_2(D)$. In that case, the right-hand side of \eqref{EvEq013a} only belongs to $L_1(D)$ and the natural functional setting to work with is $L_1(D)$, as in \cite{BS73}. 

\begin{lem}\label{EvLem019}
Given $h\in (0,1)$ and $f\in \mathcal{A}\cap L_\infty(D)$, there exists $(u,\zeta)\in \mathcal{A}\times L_1(D)$ with $u\in D(\Delta_1)$ and $(u,\zeta)$ satisfies \eqref{EvEq013}. Furthermore, there are a superlinear non-negative even and convex function $\Phi\in C^2([0,\infty))$ and a positive constant $C_0>0$ depending only on $W$ such that
\begin{equation}
\int_D \Phi(\zeta)\ \mathrm{d}x \le \frac{\Phi''(0)}{h} \frac{\|u-f\|_2^2}{h} + C_0\ . \label{EvEq020}
\end{equation}
In addition, $u\le \|f_+\|_\infty$ in $D$.
\end{lem}

\begin{proof}
For $j\ge 1$ define $W_j := W + 1/j$. Then $1/W_j$ belongs to $L_\infty(D)$ and we infer from Lemma~\ref{EvLem012} that there is $(u_j,\zeta_j)\in \mathcal{A}\times L_2(D)$ such that $u_j\in H^2(D)$ and $(u_j,\zeta_j)$ satisfies \eqref{EvEq013} with $W_j$ instead of $W$. Since $W_j\ge W$, we deduce in particular from \eqref{EvEq014} and \eqref{EvEq015} that
\begin{equation}
\|\nabla u_j\|_2^2 \le \left\| \frac{\lambda}{W_j} \right\|_1 + 2 \mathcal{E}_{W_j}(u_j) \le \left\| \frac{\lambda}{W} \right\|_1 + 2 \mathcal{E}_{W_j}(f) \le \left\| \frac{\lambda}{W} \right\|_1 + \|\nabla f\|_2^2\ . \label{EvEq021}
\end{equation} 
Also, since $f \le \|f_+\|_\infty$ in $D$, a further consequence of Lemma~\ref{EvLem012} is that
\begin{equation}
- 1 \le u_j \le \|f_+\|_\infty \;\text{ in }\; D\ . \label{EvEq027}
\end{equation}
Next, since $1/W^2\in L_1(D)$, a refined version of the de la Vall\'ee-Poussin theorem \cite{Le77} (see also \cite[Theorem~8]{La15}) guarantees that there exists a convex even function $\Phi\in C^2(\mathbb{R})$ such that $\Phi(0)=\Phi'(0)=0$, $\Phi'$ is a concave and positive function on $(0,\infty)$, and
\begin{equation}
\lim_{r\to\infty} \frac{\Phi(r)}{r} = \infty \;\;\text{ and }\;\; I_W := \int_D \Phi\left( \frac{\lambda}{W^2} \right)\ \mathrm{d}x\ . \label{EvEq022}
\end{equation}
Since $1+u_j+W_j \ge W_j \ge W$ and $\Phi$ is increasing on $[0,\infty)$, we realize that
\begin{align}
\int_D \Phi\left( - \frac{\lambda}{(1+u_j+W_j)^2} \right)\ \mathrm{d}x & = \int_D \Phi\left( \frac{\lambda}{(1+u_j+W_j)^2} \right)\ \mathrm{d}x \nonumber \\
& \le \int_D \Phi\left( \frac{\lambda}{W^2} \right)\ \mathrm{d}x = I_W\ . \label{EvEq023}
\end{align}
Owing to  \cite[Proposition~4]{BS73} and the convexity and symmetry of $\Phi$, it follows from \eqref{EvEq013a}--\eqref{EvEq013c} and \eqref{EvEq023} that 
\begin{align*}
\int_D \Phi(\zeta_j)\ \mathrm{d}x & \le \int_D \Phi\left( - \frac{u_j-f}{h} - \frac{\lambda}{2(1+u_j+W_j)^2} \right)\ \mathrm{d}x \\
& = \int_D \Phi\left( \frac{u_j-f}{h} + \frac{\lambda}{2(1+u_j+W_j)^2} \right)\ \mathrm{d}x \\
& \le \frac{1}{2} \int_D \Phi\left( 2 \frac{u_j-f}{h} \right)\ \mathrm{d}x + \frac{1}{2} \int_D \Phi\left( \frac{\lambda}{(1+u_j+W_j)^2} \right)\ \mathrm{d}x \\
& \le \frac{1}{2} \int_D \Phi\left( 2 \frac{|u_j-f|}{h} \right)\ \mathrm{d}x + \frac{I_W}{2}\ .
\end{align*}
Since the concavity of $\Phi'$ implies that $\Phi(r) \le \Phi''(0) r^2/2$ for $r\ge 0$, we end up with
\begin{equation}
\int_D \Phi(\zeta_j)\ \mathrm{d}x \le \frac{\Phi''(0)}{h} \frac{\|u_j-f\|_2^2}{h} + \frac{I_W}{2}\ . \label{EvEq024}
\end{equation}
Combining \eqref{EvEq014} and \eqref{EvEq024} gives
\begin{align*}
\int_D \Phi(\zeta_j)\ \mathrm{d}x & \le \frac{2\Phi''(0)}{h} \left[ \mathcal{E}_{W_j}(f) - \mathcal{E}_{W_j}(u_j) \right]  + \frac{I_W}{2} \\
& \le \frac{2\Phi''(0)}{h} \left[ \frac{\|\nabla f\|_2^2}{2} + \frac{\lambda}{2} \int_D \frac{\mathrm{d}x}{1+u_j+W_j} \right]  + \frac{I_W}{2}\ , 
\end{align*}
hence
\begin{equation}
\int_D \Phi(\zeta_j)\ \mathrm{d}x \le \frac{\Phi''(0)}{h} \left[ \|\nabla f\|_2^2 + \left\| \frac{\lambda}{W} \right\|_1 \right] + \frac{I_W}{2}\ . \label{EvEq025}
\end{equation}
According to \eqref{EvEq022}, the function $\Phi$ is superlinear at infinity and we infer from \eqref{EvEq025} and Dunford-Pettis' theorem that $(\zeta_j)_{j\ge 1}$ is weakly compact in $L_1(D)$. Combining this property with \eqref{EvEq021} and the compactness of the embedding of $\mathring{H}^1(D)$ in $L_2(D)$ gives a subsequence of $(u_j,\zeta_j)_{j\ge 1}$ (not relabeled) and $(u,\zeta)\in \mathring{H}^1(D)\times L_1(D)$ such that
\begin{subequations}\label{EvEq026}
\begin{align}
u_j & \rightharpoonup u \;\text{ in }\; \mathring{H}^1(D)\ , \label{EvEq026a} \\
u_j & \longrightarrow u \;\text{ in }\; L_2(D) \;\text{ and a.e. in }\; D\ , \label{EvEq026b} \\
\zeta_j & \rightharpoonup \zeta \;\text{ in }\; L_1(D)\ . \label{EvEq026c}
\end{align}
\end{subequations}
Ii follows in particular from \eqref{EvEq026b}, \eqref{EvEq027}, the square integrability of $1/W$, and Lebesgue's dominated convergence theorem that 
\begin{equation*}
- 1 \le u \le \|f_+\|_\infty \;\text{ in }\; D\ , 
\end{equation*}
and
\begin{equation}
\lim_{j\to\infty} \left\| \frac{1}{(1+u_j+W_j)^m} - \frac{1}{(1+u+W)^m} \right\|_1 = 0\ , \qquad m\in\{1,2\}\ . \label{EvEq029}
\end{equation}
Consequently, $u\in\mathcal{A} \cap L_\infty(D)$ and we may pass to the limit as $j\to\infty$ in \eqref{EvEq013a} for $u_j$ to deduce that 
\begin{equation}
\frac{u-f}{h} - \Delta u + \zeta = - \lambda g_W(u) \;\;\text{ in }\;\; \mathcal{D'}(D)\ . \label{EvEq030}
\end{equation}
However, \eqref{EvEq024}, \eqref{EvEq026c}, \eqref{EvEq029} (with $m=2$), and \eqref{EvEq030} imply that $\Delta u\in L_1(D)$, so that $(u,\zeta)$ solves \eqref{EvEq013a} in $L_1(D)$. In the same vein, we may use \eqref{EvEq026a}, \eqref{EvEq026b}, and \eqref{EvEq029} (with $m=1$) to pass to the limit as $j\to\infty$ in \eqref{EvEq014} for $u_j$ and deduce that $u$ satisfies \eqref{EvEq014}. Finally, owing to the weak convergence \eqref{EvEq026c} of $(\zeta_j)_{j\ge 1}$ in $L_1(D)$ and \eqref{EvEq026b}, a weak lower semicontinuity argument applied to \eqref{EvEq024}  based on the convexity of $\Phi$ leads to \eqref{EvEq020} with $C_0:=I_W/2$.

We are left with identifying the relation between $u$ and $\zeta$. To this end, we consider $v\in L_\infty(D)$ with $v\ge -1$ in $D$ and first observe that the weak convergence \eqref{EvEq026c} of $(\zeta_j)_{j\ge 1}$ in $L_1(D)$ and the boundedness \eqref{EvEq027} of $(v-u_j)_{j\ge 1}$ as well as its a.e. convergence \eqref{EvEq026b} allow us to apply \cite[Proposition~2.61]{FL07} and conclude that
$$
\lim_{j\to\infty} \int_D \zeta_j (v-u_j)\ \mathrm{d}x = \int_D \zeta (v-u)\ \mathrm{d}x\ .
$$
The left-hand side of the previous identity being non-positive due to \eqref{EvEq013b} for $u_j$, we realize that
$$
\int_D \zeta (v-u)\ \mathrm{d}x \le  0
$$
for all $v\in L_\infty(D)$ with $v\ge -1$ in $D$.
In particular, taking $v:=(u-1)/2\in L_\infty(D)$ which satisfies $-1<v<u$ in the set $\{x\in D\,;\, u(x)>-1\}$, we derive that $\zeta=0$ in this set. Since anyway $\zeta\le 0$ in $D$, we conclude that $\zeta\in\partial\mathbb{I}_{[-1,\infty)}(u)$ in $D$, and the proof of Lemma~\ref{EvLem019} is complete. 
\end{proof}

We also improve Lemma~\ref{saturnino} to the framework of Lemma~\ref{EvLem019}.

\begin{lem}\label{saturnino2}
Assume that $1/W\in L_r(D)$ for some $r>3d/2$ and let $f$ be a supersolution to \eqref{idefix} in the sense of Definition~\ref{DD}. Further let $u$ be the solution to \eqref{EvEq013} constructed in Lemma~\ref{EvLem019} corresponding to $1/h\ge \lambda\|1/W\|_\infty$. Then $u\le f$ in $D$, and $u$ is a supersolution to \eqref{idefix}.
\end{lem}

\begin{proof}
We keep the notation of the proof of Lemma~\ref{EvLem019}. Since $W_j\ge W$ in $D$ for all $j\ge 1$, it readily follows from Lemma~\ref{saturnino} and Remark~\ref{waters} that $u_j\le f$ in $D$ and 
$$
F_{u_j}:=-\Delta u_j+\zeta_j\ge -\lambda g_W(u_j)\quad \text{in }\ D\,.
$$
Owing to the convergences  stated in \eqref{EvEq026} and \eqref{EvEq029} we may let $j\rightarrow \infty$ in the previous two inequalities to conclude that $u\le f$ in $D$, and that $u$ is a supersolution to \eqref{idefix}.
\end{proof}

We are now in a position to prove Theorem~\ref{EvThm002}.

\begin{proof}[Proof of Theorem~\ref{EvThm002}]
Set $M:= \|(u_0)_+\|_\infty\ge 0$ and consider $h\in (0,1)$. Defining $(u_0^h, \zeta_0^h):=(u_0,0)$ we use Lemma~\ref{EvLem019} to construct by induction a sequence $(u_n^h,\zeta_n^h)_{n\ge 0}$ in $\mathcal{A}\times L_1(D)$ such that, for all $n\ge 0$, $u_{n+1}^h\in D(\Delta_1)$ and $(u_{n+1}^h,\zeta_{n+1}^h)$ solves
\begin{subequations}\label{EvEq031}
\begin{align}
\frac{u_{n+1}^h-u_n^h}{h} - \Delta u_{n+1}^h + \zeta_{n+1}^h & = - \lambda g_W(u_{n+1}^h) \;\text{ in }\;  D\ , \label{EvEq031a} \\
\zeta_{n+1}^h & \in \partial\mathbb{I}_{[-1,\infty)}(u_{n+1}^h) \;\text{ in }\; D\ , \label{EvEq031b} \\
u_{n+1}^h & = 0 \;\text{ on }\; \partial D\ . \label{EvEq031c}
\end{align}
\end{subequations}
In addition, for all $n\ge 0$, 
\begin{align}
& - 1 \le u_{n+1}^h \le M\ , \qquad x\in D\ , \label{EvEq032} \\
& \frac{1}{2h} \left\| u_{n+1}^h-u_n^h \right\|_2^2 + \mathcal{E}_W(u_{n+1}^h) \le \mathcal{E}_W(u_n^h)\ , \label{EvEq033} \\
& \int_D \Phi(\zeta_{n+1}^h)\ \mathrm{d}x \le \frac{\Phi''(0)}{h} \frac{\left\| u_{n+1}^h-u_n^h \right\|_2^2}{h} + C_0\ , \label{EvEq034}
\end{align}
the function $\Phi$ and the constant $C_0$ being defined in Lemma~\ref{EvLem019}. Introducing the time-dependent piecewise constant functions
\begin{subequations}\label{EvEq035}
\begin{align}
u^h(t,x) := \sum_{n\ge 0} u_n^h(x) \mathbf{1}_{[nh,(n+1)h)}(t)\ , \qquad (t,x)\in [0,\infty)\times D\ , \label{EvEq035a} \\ 
\zeta^h(t,x) := \sum_{n\ge 0} \zeta_n^h(x) \mathbf{1}_{[nh,(n+1)h)}(t)\ , \qquad (t,x)\in [0,\infty)\times D\ ,\label{EvEq035b}
\end{align}
\end{subequations}
we infer from \eqref{EvEq033} that, for $n\ge 0$, 
\begin{equation}
\frac{1}{2h} \sum_{m=0}^n \left\| u_{m+1}^h-u_m^h \right\|_2^2 + \mathcal{E}_W(u_{n+1}^h) \le \mathcal{E}_W(u_0)\ . \label{EvEq036}
\end{equation}
In turn, \eqref{EvEq036} gives
\begin{equation}
\frac{1}{h} \sum_{m=0}^n \left\| u_{m+1}^h-u_m^h \right\|_2^2 + \left\| \nabla u_{n+1}^h \right\|_2^2 \le C_1 := 2\mathcal{E}_W(u_0) + \left\| \frac{1}{W} \right\|_1\ , \qquad n\ge 0\ . \label{EvEq037}
\end{equation}
Now, combining \eqref{EvEq034} and \eqref{EvEq037} leads us to 
$$
\sum_{m=0}^{n+1} \int_D \Phi(\zeta_m^h)\ \mathrm{d}x = \sum_{m=0}^{n} \int_D \Phi(\zeta_{m+1}^h)\ \mathrm{d}x \le C_1 \frac{\Phi''(0)}{h} + (n+1) C_0\ ,
$$
hence
\begin{equation}
\sum_{m=0}^{n+1} h \int_D \Phi(\zeta_m^h)\ \mathrm{d}x \le C_1 \Phi''(0) + C_0 (n+1)h\ , \qquad n\ge 0\ . \label{EvEq038}
\end{equation}

Let us now fix $t>0$ and translate the above derived estimates in terms of $u^h$ and $\zeta^h$. Since $t\in [(n+1)h,(n+2)h)$ for some $n\ge -1$, it follows from \eqref{EvEq032} and \eqref{EvEq037} for $n\ge 0$ and from the definition of $u^h$ for $n=-1$ that
\begin{equation}
- 1 \le u^h(t) \le M \;\text{ in }\; D\ , \qquad \|\nabla u^h(t)\|_2^2 \le C_1\ . \label{EvEq039}
\end{equation}
Furthermore, by \eqref{EvEq038} for $n\ge 0$ and the definition of $\zeta^h$ for $n=-1$,
\begin{align}
\int_0^t \int_D \Phi(\zeta^h(\tau,x))\ \mathrm{d}x\mathrm{d}\tau & \le \int_0^{(n+2)h} \int_D \Phi(\zeta^h(\tau,x))\ \mathrm{d}x\mathrm{d}\tau \nonumber \\
& = \sum_{m=0}^{n+1} \int_{mh}^{(m+1)h} \int_D \Phi(\zeta^h(\tau,x))\ \mathrm{d}x\mathrm{d}\tau \nonumber \\
& = \sum_{m=0}^{n+1} h \int_D \Phi(\zeta_m^h(x))\ \mathrm{d}x \nonumber \\
& \le C_1 \Phi''(0) + C_0 t\ . \label{EvEq040}
\end{align}
We finally deduce from \eqref{EvEq035a}, \eqref{EvEq037}, and Cauchy-Schwarz' inequality  that
\begin{equation}
\|u^h(t)-u^h(s)\|_2 \le \sqrt{C_1} \sqrt{t-s+h}\ , \qquad s\in [0,t]\ . \label{EvEq041} 
\end{equation}

Owing to the compactness of the embedding of $\mathring{H}^1(D)$ in $L_2(D)$, the estimates \eqref{EvEq039} and \eqref{EvEq041} allow us to apply the variant of Arzel\`a-Ascoli theorem stated in \cite[Proposition~3.3.1]{AGS08} to obtain the existence of $u\in C([0,\infty);L_2(D))$ and a sequence $(h_k)_{k\ge 1}$ of positive real numbers such that
\begin{equation}
\lim_{k\to\infty} h_k = 0\ , \qquad \lim_{k\to\infty} \|u^{h_k}(t) - u(t)\|_2 = 0 \;\text{ for all }\; t\ge 0\ . \label{EvEq042}
\end{equation}
In addition, for all $T>0$, the superlinearity \eqref{EvEq022} of $\Phi$, the bound \eqref{EvEq040}, and Dunford-Pettis' theorem guarantee that $(\zeta^h)_h$ is relatively weakly sequentially compact in $L_1((0,T)\times D)$ while $(\nabla u^h)_h$ is obviously relatively weakly compact in $L_2((0,T)\times D;\mathbb{R}^d)$ according to \eqref{EvEq039}. We may thus further assume that there is $\zeta\in L_1((0,T)\times D)$ such that
\begin{align}
\zeta^{h_k} & \rightharpoonup \zeta \;\;\text{ in }\;\; L_1((0,T)\times D)\ , \label{EvEq043} \\
\nabla u^{h_k} & \rightharpoonup \nabla u \;\;\text{ in }\;\; L_2((0,T)\times D;\mathbb{R}^d)\ , \label{EvEq044} \\
u^{h_k} & \longrightarrow u \;\;\text{ a.e. in }\;\; (0,T)\times D\ . \label{EvEq045}
\end{align}
A first consequence of \eqref{EvEq039}, \eqref{EvEq044}, and \eqref{EvEq045} is that $u(t)\in \mathcal{A}$ and satisfies \eqref{EvEq003} for all $t\ge 0$. Next, \eqref{EvEq042}, \eqref{EvEq045}, the square integrability of $1/W$, and Lebesgue's dominated convergence theorem entail that
\begin{equation}
\lim_{k\to\infty} \int_0^t \int_D \frac{\mathrm{d}x\mathrm{d}\tau}{(1+u^{h_k}(\tau,x) + W(x))^2} = \int_0^t \int_D \frac{\mathrm{d}x\mathrm{d}\tau}{(1+u(\tau,x)+ W(x))^2}\ , \qquad t\ge 0\ . \label{EvEq046}
\end{equation}
Also, we argue as in the proof of Lemma~\ref{EvLem019} to deduce from \eqref{EvEq039}, \eqref{EvEq043}, and \eqref{EvEq045} that $u$ and $\zeta$ are  related according to Definition~\ref{EvDef001}~(c).

Let us now identify the equation solved by $(u,\zeta)$. To this end, consider $\vartheta\in \mathring{H}^1(D)\cap L_\infty(D)$ and $t>0$. For $k$ large enough,  there is $n_k\ge 0$ such that $t\in [(n_k+1)h_k,(n_k+2)h_k)$ and a classical computation relying on \eqref{EvEq031a}, \eqref{EvEq031b}, and \eqref{EvEq035} gives (recalling that $(u_0^{h_k},\zeta_0^{h_k})=(u_0,0)$)
\begin{align*}
\int_D (u^{h_k}(t)-u_0) \vartheta\ \mathrm{d}x & = - \int_0^t \int_D \left[ \nabla u^{h_k}\cdot \nabla\vartheta + \zeta^{h_k} \vartheta + \frac{\lambda \vartheta}{2(1+u^{h_k}+W)^2} \right]\ \mathrm{d}x\mathrm{d}\tau \\
& \qquad + \int_0^{h_k} \int_D \left[ \nabla u_0\cdot \nabla\vartheta + \frac{\lambda \vartheta}{2(1+u_0+W)^2} \right]\ \mathrm{d}x\mathrm{d}\tau \\
& \qquad + \int_t^{(n_k+2)h_k} \int_D \left[ \nabla u^{h_k}\cdot \nabla\vartheta + \zeta^{h_k} \vartheta + \frac{\lambda \vartheta}{2(1+u^{h_k}+W)^2} \right]\ \mathrm{d}x\mathrm{d}\tau\ .
\end{align*}
Thanks to the convergences \eqref{EvEq042}, \eqref{EvEq043}, \eqref{EvEq044}, and \eqref{EvEq046}, we may pass to the limit as $k\to\infty$ in the previous identity and deduce that $(u,\zeta)$ solves \eqref{EvEq001a}, \eqref{EvEq001b} in the weak sense \eqref{EvEq011}.

We are left with passing to the limit in the discrete energy inequality \eqref{EvEq036}. For $t>0$ and $k$ large enough, there is $n_k\ge 0$ such that $t \in [(n_k+1)h_k,(n_k+2)h_k)$. Then $u^{h_k}(t) = u_{n_k+1}^{h_k}$ and the discrete energy inequality \eqref{EvEq036} reads
\begin{equation}
\frac{1}{2h_k} \sum_{m=0}^{n_k} \left\| u_{m+1}^{h_k} - u_m^{h_k} \right\|_2^2 + \mathcal{E}_W(u^{h_k}(t)) \le \mathcal{E}_W(u_0)\ . \label{EvEq047}
\end{equation}
On the one hand, we infer from \eqref{EvEq039} that $(u^{h_k}(t))_{k\ge 1}$ is bounded in $\mathring{H}^1(D)\cap L_\infty(D)$ and converges towards $u(t)$ in $L_2(D)$. We may thus extract a subsequence of $(u^{h_k}(t))_{k\ge 1}$ (possibly depending on $t$) which converges weakly towards $u(t)$ in $\mathring{H}^1(D)$ as well as a.e. in $D$. These properties along with the integrability of $1/W$ and Lebesgue's dominated convergence theorem readily imply that
\begin{equation}
\mathcal{E}_W(u(t)) \le \liminf_{k\to\infty} \mathcal{E}_W(u^{h_k}(t))\ . \label{EvEq048}
\end{equation}
Also, for $\delta>h_k$,
\begin{align}
\frac{1}{2h_k} \sum_{m=0}^{ n_k} \left\| u_{m+1}^{h_k} - u_m^{h_k} \right\|_2^2 & = \frac{1}{2h_k^2} \sum_{m=0}^{ n_k} \int_{m h_k}^{(m+1) h_k} \left\| u_{m+1}^{h_k} - u_m^{h_k} \right\|_2^2\ \mathrm{d}\tau \nonumber \\
& = \frac{1}{2} \int_0^{(n_k+1) h_k} \left\| \frac{u^{h_k}(\tau+h_k) - u^{h_k}(\tau)}{h_k} \right\|_2^2 \ \mathrm{d}\tau \nonumber \\
& \ge \frac{1}{2} \int_0^{t-\delta} \left\| \frac{u^{h_k}(\tau+  h_k) - u^{h_k}(\tau)}{h_k} \right\|_2^2 \ \mathrm{d}\tau\ . \label{EvEq049}
\end{align}
Let $\delta\in (0,t/2)$. Since 
$$
(\tau,x) \longmapsto \frac{u^{h_k}(\tau+h_k) - u^{h_k}(\tau)}{h_k}
$$
converges to $\partial_t u$ in $\mathcal{D}'((0,\infty)\times D)$ by \eqref{EvEq042} and is bounded in $L_2((0,t-\delta)\times D)$ due to \eqref{EvEq037} and \eqref{EvEq049}, we realize that $\partial_t u$ belongs to $L_2((0,t-\delta)\times D)$ and satisfies 
$$
\frac{1}{2}\int_0^{t-\delta} \|\partial_t u(\tau)\|_2^2\ \mathrm{d}\tau \le \liminf_{k\to \infty} \frac{1}{2h_k} \sum_{m=0}^{n_k} \left\| u_{m+1}^{h_k} - u_m^{h_k} \right\|_2^2\ .
$$
Since $\delta\in (0,t/2)$ was arbitrarily chosen, the previous inequality is also valid for $\delta=0$ and we combine it with \eqref{EvEq047} and \eqref{EvEq048} to conclude that $u$ satisfies the energy inequality \eqref{EvEq010}. 

Finally, if there are $\kappa\in (0,1)$ and $T>0$ such that $u\ge \kappa-1$ in $(0,T)\times D$, then $g_W(u) \le 1/(2\kappa^2)$ and thus belongs to $L_\infty(0,T)\times D)$. Classical parabolic regularity results then complete the proof of Theorem~\ref{EvThm002}.
\end{proof}

We immediately derive the time monotonicity of the just constructed solution  when the initial value is a supersolution to \eqref{idefix}.

\begin{prop}\label{jovanotti}
Let $1/W\in L_r(D)$ for some $r>3d/2$ and let $u_0\in \mathcal{A}\cap L_\infty(D)$ be a supersolution to \eqref{idefix}. If $u$ denotes the corresponding solution to \eqref{EvEq001} constructed in Theorem~\ref{EvThm002}, then for a.a. $x\in D$ the function $t\mapsto u(t,x)$ is non-increasing.
\end{prop}

\begin{proof}
We keep the notation of the proof of Theorem~\ref{EvThm002}. Thanks to Lemma~\ref{saturnino2} and the assumption on the initial value $u_0$, an induction argument entails that $u_n^h\ge u_{n+1}^h$ in $D$ and $u_{n+1}^h$ is a supersolution to \eqref{idefix} for all $n\in \N$ provided that $h$ is small enough. Therefore, the function $t\mapsto u^h(t,x)$ is non-increasing for a.a. $x\in D$ and the assertion follows from \eqref{EvEq042}.
\end{proof}

We next focus on the uniqueness of energy solutions when $1/W$ enjoys suitable integrability properties and actually prove a comparison principle. 

\begin{proof}[Proof of Theorem~\ref{EvThm005}: Uniqueness and comparison principle]
Let $1/W\in L_r(D)$ with $r>3d/2$. Setting $p:=2r/(r-3)$ when $r>3$ and $p:=\infty$ when $r\le 3$ (the latter being possible only in one space dimension $d=1$), the constraint $r>3d/2$ guarantees that $p\in (2,2^*)$ when it is finite, the Sobolev exponent $2^*$ being given by $2^*:=2d/(d-2)$ for $d\ge 3$ and $2^*:=\infty$ for $d\in\{1,2\}$. This choice of $p$ implies the validity of the Gagliardo-Nirenberg inequality
$$
\|w\|_p \le C \|\nabla w\|_2^\theta \|w\|_2^{1-\theta}\ , \qquad w\in \mathring{H}^1(D)\ ,
$$
where $\theta := d(p-2)/2p \in (0,1)$ for $p$ finite and $\theta:=1/2$ for $p=\infty$ and $C$ depends only on $D$, $d$, and $p$. 

Now, let $u_0, v_0\in\mathcal{A}\cap L_\infty(D)$ and consider two energy solutions $u$ and $v$ to \eqref{EvEq001} in the sense of Definition~\ref{EvDef001} with initial  values $u_0$ and $v_0$, respectively. Owing to \eqref{EvEq001b}, the boundedness of $v-u$, and the integrability of $\zeta_v-\zeta_u$, there holds
$$
\int_D (\zeta_v-\zeta_u)(v-u)_+\ \mathrm{d}x\ge 0\ ,
$$
and we infer from \eqref{EvEq001} that
\begin{align*}
\frac{1}{2} \frac{d}{dt} \|(v-u)_+\|_2^2 & \le - \|\nabla(v-u)_+\|_2^2 + \frac{\lambda}{2} \int_D \frac{(1+u+W+1+v+W)(v-u)_+^2}{(1+u+W)^2(1+v+W)^2}\ \mathrm{d}x \\
& \le - \|\nabla(v-u)_+\|_2^2 + \lambda \int_D \frac{(v-u)_+^2}{W^3}\ \mathrm{d}x\ .
\end{align*}
We next use H\"older's inequality along with the previously recalled Gagliardo-Nirenberg inequality to obtain
\begin{align*}
\frac{1}{2} \frac{d}{dt} \|(v-u)_+\|_2^2 & \le - \|\nabla(v-u)_+\|_2^2 + \lambda \|(v-u)_+\|_p^2 \|W^{-1}\|_r^3 \\
& \le - \|\nabla(v-u)_+\|_2^2 + \lambda C^2 \|W^{-1}\|_r^3 \|\nabla(v-u)_+\|_2^{2\theta} \|(v-u)_+\|_2^{2(1-\theta)}\ .
\end{align*}
We finally deduce from Young's inequality that
\begin{align*}
\frac{1}{2} \frac{d}{dt} \|(v-u)_+\|_2^2 & \le (\theta-1) \|\nabla(v-u)_+\|_2^2 + (1-\theta) \left[ \lambda C^2 \|W^{-1}\|_r^3 \right]^{1/(1-\theta)} \|(v-u)_+\|_2^2 \\
& \le \left[ \lambda C^2 \|W^{-1}\|_r^3 \right]^{1/(1-\theta)} \|(v-u)_+\|_2^2\ .
\end{align*}
Integrating the previous differential inequality gives
\begin{equation}
\|(v-u)_+(t)\|_2^2 \le e^{C_2 t} \|(v_0-u_0)_+\|_2^2 \ , \qquad t\ge 0\ , \qquad C_2 := \left[ \lambda C^2 \|W^{-1}\|_r^3 \right]^{1/(1-\theta)}\ . \label{EvEq050}
\end{equation}
On the one hand, it readily follows from \eqref{EvEq050} that, if $u_0\le v_0$ a.e. in $D$, then $u(t)\le v(t)$ a.e. in $D$ for all $t\ge 0$. On the other hand, using again \eqref{EvEq050}, we realize that
\begin{align*}
\|(v-u)(t)\|_2^2 & = \|(v-u)_+(t)\|_2^2 + \|(u-v)_+(t)\|_2^2 \\
& \le e^{C_2 t} \left( \|(v_0-u_0)_+\|_2^2 + \|(u_0-v_0)_+\|_2^2 \right) = e^{C_2 t} \|v_0-u_0\|_2^2\ ,
\end{align*}
hence the claimed uniqueness. 
\end{proof}

\section{The evolution problem: large time dynamics}\label{Sec6}

We now investigate the large time behavior of energy solutions by characterizing the $\omega$-limit sets as stated in Theorem~\ref{EvThm004}.

\begin{proof}[Proof of Theorem~\ref{EvThm004}]
 Fix $u_0\in \mathcal{A}\cap L_\infty(D)$ and consider an energy solution $(u,\zeta)$ satisfying \eqref{EvEq003} as provided by Theorem~\ref{EvThm002}. The energy inequality \eqref{EvEq010} and the square integrability of $1/W$ imply then that
$$
\int_0^t \|\partial_t u(\tau)\|_2^2\ \mathrm{d}\tau + \|\nabla u(t)\|_2^2 \le C_1 = 2 \mathcal{E}_W(u_0) + \left\| \frac{\lambda}{W} \right\|_1\ , \qquad t\ge 0\ ,
$$
and thus
\begin{equation}
\int_0^\infty \|\partial_t u(\tau)\|_2^2\ \mathrm{d}\tau + \sup_{t\ge 0}\left\{ \|\nabla u(t)\|_2^2 \right\} \le C_1\ . \label{LDEq001}
\end{equation}
By \eqref{LDEq001}, $(u(t))_{t\ge 0}$ is bounded in $\mathring{H}^1(D)$ and thus relatively compact in $L_2(D)$. Consequently, there are a sequence $(t_k)_{k\ge 1}$ of positive real numbers in $(1,\infty)$ and $v\in L_2(D)$ such that
\begin{equation}
\lim_{k\to\infty} t_k = \infty\ , \qquad \lim_{k\to\infty} \|u(t_k)-v\|_2=0\ . \label{LDEq002}
\end{equation}
We now define the sequences $(V_k,\xi_k)_{k\ge 1}$ by
$$
(V_k,\xi_k)(s,x) := (u,\zeta)(s+t_k,x)\ , \qquad (s,x)\in [-1,1]\times D\,,
$$
where
$$
\zeta:=-\lambda g_W(u)+\Delta u-\partial_t u\,.
$$
On the one hand, it readily follows from \eqref{LDEq001} that 
$$
(V_k)_{k\ge 1} \text{ is bounded in } L_\infty(-1,1;\mathring{H}^1(D)) \text{ and in } W_2^1(-1,1;L_2(D))\ .
$$ 
Owing to the compactness of the embedding of $\mathring{H}^1(D)$ in $L_2(D)$, we infer from \cite[Corollary~4]{Si87} that there are $V\in C([-1,1];L_2(D))$ and a subsequence of $(V_k)_{k\ge 1}$ (not relabeled) such that
\begin{align}
V_k & \longrightarrow V \;\text{ in }\; C([-1,1];L_2(D)) \;\text{ and a.e. in }\; (-1,1)\times D\ , \label{LDEq003} \\
\nabla V_k & \rightharpoonup \nabla V \;\text{ in }\; L_2((-1,1)\times D;\mathbb{R}^d)\ . \label{LDEq004}
\end{align}
A first consequence of \eqref{LDEq001} and \eqref{LDEq003} is that, for all $s\in [-1,1]$, 
\begin{align*}
\|V(s) - v\|_2 & = \lim_{k\to\infty} \|V_k(s) - V_k(0)\|_2 \le \lim_{k\to\infty} \left| \int_{t_k}^{s+t_k} \|\partial_t u(\tau)\|_2\ \mathrm{d}\tau \right| \\ 
& \le \sqrt{2} \lim_{k\to\infty} \left( \int_{-1+t_k}^{1+t_k} \|\partial_t u(\tau)\|_2^2\ \mathrm{d}\tau \right)^{1/2} = 0\ ,
\end{align*}
so that
\begin{equation}
V(s) \equiv v\ , \qquad s\in [-1,1]\ . \label{LDEq005}
\end{equation}
Another consequence of \eqref{LDEq003}, the square integrability of $1/W$, and Lebesgue's dominated convergence theorem is that
\begin{equation}
\lim_{k\to\infty} \int_{-1}^1 \int_D g_W(V_k) \mathrm{d}x \mathrm{d}s = \int_{-1}^1 \int_D  g_W(V)\mathrm{d}x \mathrm{d}s\ . \label{LDEq006}
\end{equation}

On the other hand, we introduce
$$
G_k := - \lambda g_W(V_k) - \partial_s V_k\ , \qquad k\ge 1\ .
$$
Since $(\partial_s V_k)_{k\ge 1}$ is bounded in $L_2((-1,1)\times D)$ by \eqref{LDEq001} and 
$$
\left| \lambda g_W(V_k) \right| \le \frac{\lambda}{2W^2} \in L_1((-1,1)\times D)\ ,
$$
the sequence $(G_k)_{k\ge 1}$ is the sum of two sequences which are relatively weakly sequentially compact in $L_1((-1,1)\times D)$ and is thus also relatively weakly sequentially compact in $L_1((-1,1)\times D)$. Using once more the de la Vall\'ee-Poussin theorem \cite{Le77, La15}, there is a non-negative and even convex function $\Phi\in C^2(\mathbb{R})$ such that
\begin{equation}
\lim_{r\to \infty} \frac{\Phi(r)}{r} = \infty\ , \qquad C_2 := \sup_{k\ge 1} \left\{ \int_{-1}^1 \int_D \Phi(|G_k|)\ \mathrm{d}x\mathrm{d}s \right\} < \infty\ . \label{LDEq008}
\end{equation} 
Furthermore, the regularity of $(u,\zeta)$ ensures that, for almost every $s\in [-1,1]$, $V_k(s)\in \mathring{H}^1(D)$, $\Delta V_k(s)\in L_1(D)$, $\xi_k(s)\in L_1(D)$, and $G_k(s) \in L_1(D)$. Together with the weak formulation \eqref{EvEq011} of \eqref{EvEq001} and \cite[Theorem~1]{BS73}, these properties imply that $V_k(s)$ is the unique solution to 
\begin{subequations}\label{LDEq007}
\begin{align}
-\Delta V_k(s) + \xi_k(s) & = G_k(s) \;\text{ in }\; D\ , \label{LDEq007a} \\
\xi_k(s) & \in \partial\mathbb{I}_{[-1,\infty)}(V_k(s)) \;\text{ in }\; D\ , \label{LDEq007b} \\
V_k(s) & = 0 \;\text{ on }\; D\ . \label{LDEq007c}
\end{align}
\end{subequations}
We then apply \cite[Proposition~4]{BS73} to deduce from \eqref{LDEq007} that, for all $k\ge 1$ , 
$$
\int_D \Phi(\xi_k(s,x))\ \mathrm{d}x \le \int_D \Phi(G_k(s,x))\ \mathrm{d}x \;\text{ for almost every }\; s\in (-1,1)\ ,
$$
hence, thanks to \eqref{LDEq008},
$$
\int_{-1}^1 \int_D \Phi(\xi_k(s,x))\ \mathrm{d}x\mathrm{d}s \le C_2\ .
$$
The superlinearity \eqref{LDEq008} of $\Phi$ along with the previous bound and Dunford-Pettis' theorem entail that $(\xi_k)_{k\ge 1}$ is relatively weakly sequentially compact in $L_1((-1,1)\times D)$. Consequently, there are $\xi\in L_1((-1,1)\times D)$ and a subsequence of $(\xi_k)_{k\ge 1}$ (not relabeled) such that 
\begin{equation}
\xi_k \rightharpoonup \xi \;\text{ in }\; L_1((-1,1)\times D)\ . \label{LDEq009}
\end{equation}
Now, to identify the equation solved by $(V,\xi)$, we infer from the weak formulation \eqref{EvEq011} of \eqref{EvEq001} that, for $\vartheta\in\mathcal{A}\cap L_\infty(D)$ and $k\ge 1$,
\begin{align*}
& \int_{-1}^1 \int_D \left[ \nabla V_k \cdot \nabla\vartheta + \xi_k \vartheta + \lambda\vartheta g_W(V_k) \right]\ \mathrm{d}x\mathrm{d}s \\
& = \int_{-1+t_k}^{1+t_k} \int_D \left[ \nabla u \cdot \nabla\vartheta + \zeta \vartheta + \lambda\vartheta g_W(u) \right]\ \mathrm{d}x\mathrm{d}\tau \\
& = \int_D [u(-1+t_k)-u(1+t_k)] \vartheta\ \mathrm{d}x  = \int_{-1+t_k}^{1+t_k} \int_D \vartheta \partial_t u\ \mathrm{d}x\mathrm{d}\tau\ . 
\end{align*}
By Cauchy-Schwarz' inequality,
\begin{align*}
\left| \int_{-1+t_k}^{1+t_k} \int_D \vartheta \partial_t u\ \mathrm{d}x\mathrm{d}\tau \right| \le \|\vartheta\|_2 \int_{-1+t_k}^{1+t_k} \|\partial_t u\|_2\ \mathrm{d}\tau \le \sqrt{2} \|\vartheta\|_2 \left( \int_{-1+t_k}^{1+t_k} \|\partial_t u\|_2^2\ \mathrm{d}\tau \right)^{1/2}\ ,
\end{align*}
and the right-hand side of the above inequality converges to zero as $k\to\infty$ by \eqref{LDEq001}. Consequently,
\begin{equation}
\lim_{k\to\infty} \int_{-1}^1 \int_D \left[ \nabla V_k \cdot \nabla\vartheta + \xi_k \vartheta + \lambda\vartheta g_W(V_k) \right]\ \mathrm{d}x\mathrm{d}s = 0\ . \label{LDEq010}
\end{equation}

Gathering \eqref{LDEq004}, \eqref{LDEq005}, \eqref{LDEq006}, \eqref{LDEq009}, and \eqref{LDEq010}, we end up with
$$
\int_{-1}^1 \int_D \left[ \nabla v \cdot \nabla\vartheta + \xi \vartheta + \lambda\vartheta g_W(v) \right]\ \mathrm{d}x\mathrm{d}s = 0\ ,
$$
which entails, in particular, that $\xi$ does not depend on time and that $\Delta v = \xi + \lambda g_W(v)$ belongs to $L_1(D)$. We finally check that $\xi\in \partial\mathbb{I}_{[-1,\infty)}(v)$ a.e. in $D$ as in the proof of Lemma~\ref{EvLem012}, recalling that $-1 \le v \le \|(u_0)_+\|_\infty$ as a consequence of \eqref{EvEq003} and \eqref{LDEq002}. Also, $v$ belongs to $\omega(u_0)$ by \eqref{LDEq002}. Thus, $\omega(u_0)$ is non-empty and obviously bounded in $\mathring{H}^1(D)$ by \eqref{LDEq001} and Poincar\'e's inequality.

To finish off the proof, let us assume that $u_0\ge U_\lambda$. The comparison principle in Theorem~\ref{EvThm005} implies that $u(t)\ge U_\lambda$ for all $t\ge 0$. Consequently, if $v\in \omega(u_0)$, then $v\ge U_\lambda$ and the maximality of $U_\lambda$ stated in Proposition~\ref{P1} entails that $v=U_\lambda$ as claimed.
\end{proof}

Combining Proposition~\ref{jovanotti} and Theorem~\ref{EvThm004} gives several properties of the solution to \eqref{EvEq001} starting from the rest state $u_0=0$ as summarized in Theorem~\ref{obelix888}.

\begin{proof}[Proof of Theorem~\ref{obelix888}] Let $u_0=0$ and let $u$ be the corresponding solution to \eqref{EvEq001} given by Theorem~\ref{EvThm002}
and Theorem~\ref{EvThm005}. Clearly, $u_0$ is a supersolution to \eqref{idefix} and satisfies $u_0\ge U_\lambda$ in $D$. It then readily follows from Proposition~\ref{jovanotti} and Theorem~\ref{EvThm004} that $u(t_1)\ge u(t_2)\ge U_\lambda$ in $D$ for $t_1<t_2$. On the one hand, this ordering property obviously implies that 
\begin{equation}\label{gilmour}
\mathcal{C}(u(t_1))\subset \mathcal{C}(u(t_2))\subset \mathcal{C}(U_\lambda)\,.
\end{equation}
Since the measure of $\mathcal{C}(U_\lambda)$ equals 0 if $\lambda <\Lambda_z$, statements~(i) and~(ii) follow. 

As for statement~(iii), let $\varphi_1$ be the positive eigenfunction  to $-\Delta_1$ associated with the first eigenvalue $\mu_1$ and normalized as $\|\varphi_1\|_1=1$. It then follows from \eqref{EvEq011} that 
\begin{equation*}
 \int_0^t \int_D \zeta_u \varphi_1 \, \mathrm{d}x\mathrm{d}s =
-\int_D u(t) \varphi_1\ \mathrm{d}x  -\mu_1 \int_0^t \int_D  u\varphi_1 \ \mathrm{d}x\mathrm{d}s - \lambda \int_0^t\int_D \varphi_1 g_W(u)\ \mathrm{d}x\mathrm{d}s\,.
\end{equation*}
Since $-1\le u(s)\le 0$ in $D$ for all $s\ge 0$, we further obtain that
\begin{equation*}
 \int_0^t \int_D \zeta_u \varphi_1 \, \mathrm{d}x\mathrm{d}s \le 1 -\left(\lambda \int_D \varphi_1 g_W(0)\,\rd x-\mu_1\right) t\,.
\end{equation*}
Introducing
$$
\Lambda^*:=\frac{\mu_1}{\int_D \varphi_1 g_W(0)\,\rd x}
$$
and setting 
$$
T_z:=\frac{\Lambda^*}{\mu_1(\lambda-\Lambda^*)}\,,
$$
we realize that 
\begin{equation}\label{neville}
\int_0^t \int_D \zeta_u \varphi_1 \, \mathrm{d}x\mathrm{d}s <0
\end{equation}
for all $t>T_z$. Consequently, given $t>T_z$ there is $s(t)\in (0,t)$ such that $\zeta_{u}(s(t))\not\equiv 0$ and thus $\vert\mathcal{C}(u(s(t))\vert>0$. The time monotonicity \eqref{gilmour} of the coincidence set then implies that $\vert\mathcal{C}(u(t))\vert>0$ and the proof of Theorem~\ref{obelix888}.

To prove statement (iv) we proceed along the lines of \cite{BCMR,GhG08a} and construct a subsolution to \eqref{EvEq001} {\tred for} $\lambda>\Lambda_z$ which is well-separated from $-1$. This will eventually imply that the corresponding maximal stationary solution is unzipped, contradicting the assumption that $\lambda>\Lambda_z$. More specifically, set $M:= \|W\|_\infty$ and consider $\lambda>\Lambda_z$. Let $u$ be the solution to \eqref{EvEq001} with initial value $u_0=0$ and assume for contradiction that $\zeta_u(t)\equiv 0$ for all $t>0$. For $\varepsilon\in (0,1)$, we define the function 
\begin{equation}
\Psi_\varepsilon(r) := -1-M + \left[ \frac{(1+r+M)^3 - \varepsilon^3 (1+M)^3}{1-\varepsilon^3} \right]^{1/3}\ , \qquad r\in (r_\varepsilon, \infty)\ , \label{acdc1}
\end{equation}
with $r_\varepsilon := -(1+M)(1-\varepsilon)<0$. Observe that
\begin{equation}
\Psi_\varepsilon'(r) = \frac{(1+r+M)^2}{(1-\varepsilon^3)^{1/3} \left[ (1+r+M)^3 - \varepsilon^3 (1+M)^3 \right]^{2/3}} > 0\ , \qquad r\in (r_\varepsilon,\infty)\ , \label{acdc2}
\end{equation}
and
$$
\Psi_\varepsilon''(r) = - \frac{2 \varepsilon^3 (1+M)^3}{(1-\varepsilon^3)^{1/3}} \frac{(1+r+M)}{\left[ (1+r+M)^3 - \varepsilon^3 (1+M)^3 \right]^{5/3}} < 0\ , \qquad r\in (r_\varepsilon,\infty)\ ,
$$
so that $\Psi_\varepsilon$ is an increasing concave function from $(r_\varepsilon,\infty)$ onto $(-(1+M),\infty)$. We next define $v_\varepsilon := \Psi_\varepsilon^{-1}(u)$ in $(0,\infty)\times D$. Since $\Psi_\varepsilon$ is increasing and $u$ ranges in $[-1,0]$,  we obtain that
\begin{equation}
\varrho_\varepsilon := \left[ \varepsilon^3 (1+M)^3 + (1-\varepsilon^3) M^3 \right]^{1/3} - (1+M) \le v_\varepsilon \le 0 \;\;\text{ in }\;\; (0,\infty)\times D\ . \label{acdc3}
\end{equation}
Observe that the convexity of $r\mapsto r^3$ ensures that 
\begin{equation}\label{notadam}
\varrho_\varepsilon\ge \varepsilon^3-1>-1\,.
\end{equation}
We next infer from \eqref{EvEq001a} that $v_\varepsilon$ solves
\begin{equation}
\partial_t v_\varepsilon - \Delta v_\varepsilon = \frac{\Psi_\varepsilon''(v_\varepsilon)}{\Psi_\varepsilon'(v_\varepsilon)} |\nabla v_\varepsilon|^2 - \lambda (1-\varepsilon^3) g_W(v_\varepsilon) + \frac{\lambda}{2} \frac{S_\varepsilon(v_\varepsilon)}{H_\varepsilon(v_\varepsilon)} \;\text{ in }\; (0,\infty)\times D\ , \label{acdc4}
\end{equation}
where
$$
S_\varepsilon (r,x) := (1-\varepsilon^3) \Psi_\varepsilon'(r) \left[ 1+ \Psi_\varepsilon(r) + W(x) \right]^2 - \left[ 1+ r + W(x) \right]^2\ , \qquad (r,x)\in [\varrho_\varepsilon,0] \times D\,,
$$
and
$$
H_\varepsilon (r,x) :=\Psi_\varepsilon'(r) \left[ 1+ \Psi_\varepsilon(r) + W(x) \right]^2 \left[ 1+ r + W(x) \right]^2 >0\ , \qquad (r,x)\in [\varrho_\varepsilon,0] \times D\,.
$$
It follows from the definition \eqref{acdc1} of $\Psi_\varepsilon$ that, for $r\in [\varrho_\varepsilon,0]$,
\begin{equation}
S_\varepsilon (r,x) = (M-W(x)) \left[ 1 - R_\varepsilon(r) \right] \left[ 2 + 2r + M + W(x) - (M-W(x)) R_\varepsilon(r) \right]\ , \label{acdc5}
\end{equation}
where
$$
R_\varepsilon(r) := \frac{(1-\varepsilon^3)^{1/3} (1+r+M)}{\left[ (1+r+M)^3 - \varepsilon^3 (1+M)^3 \right]^{1/3}}\ .
$$
Since 
$$
R_\varepsilon'(r) = - \frac{\varepsilon^3  (1-\varepsilon^3)^{1/3} (1+M)^3}{\left[ (1+r+M)^3 - \varepsilon^3 (1+M)^3 \right]^{4/3}} \le 0\ , \qquad r\in [\varrho_\varepsilon,0]\ ,
$$
there holds $R_\varepsilon(0)\le  R_\varepsilon(r) \le R_\varepsilon(\varrho_\varepsilon)$ for $r\in [\varrho_\varepsilon,0]$, hence
\begin{equation}
1 \le R_\varepsilon(r) \le \frac{\varrho_\varepsilon + 1+ M}{M}\ , \qquad r\in [\varrho_\varepsilon,0]\ . \label{acdc6}
\end{equation}
Consequently, for $r\in [\varrho_\varepsilon,0]$ and $x\in D$, it follows from \eqref{acdc6} and the definition of $M$ that
\begin{align}
& 2 + 2r + M + W(x) - (M-W(x)) R_\varepsilon(r) \nonumber \\
& \qquad \ge 2 + 2 \varrho_\varepsilon + M + W(x) - \frac{M-W(x)}{M} (\varrho_\varepsilon + 1+ M) \nonumber \\
& \qquad = \frac{M+W(x)}{M} (1+\varrho_\varepsilon) + 2 W(x) \ge 0\ . \label{acdc7}
\end{align}
We then infer from \eqref{acdc5}, \eqref{acdc6}, \eqref{acdc7}, and the definition of $M$ that $S_\varepsilon(r,x)\le 0$ for $(r,x)\in  [\varrho_\varepsilon,0]\times D$. Along with \eqref{acdc4} and the monotonicity and concavity of $\Psi_\varepsilon$, this readily implies that 
\begin{subequations}\label{acdc8}
\begin{equation}
\partial_t v_\varepsilon - \Delta v_\varepsilon \le - \lambda (1-\varepsilon^3) g_W(v_\varepsilon) \;\text{ in }\; (0,\infty)\times D\,, \label{acdc8a}
\end{equation}
while \eqref{acdc3} and \eqref{notadam} entail that $g_W(v_\ve)$ belongs to $L_\infty((0,\infty)\times D)$. 
In addition, since $\Psi_\varepsilon(0)=0$, it follows from \eqref{EvEq001c} and \eqref{EvEq001d} that 
\begin{align}
v_\varepsilon & = 0 \;\text{ on }\; (0,\infty)\times \partial D\ ,\label{acdc8b} \\
v_\varepsilon(0) & = 0 \;\text{ in }\; D\ . \label{acdc8c}
\end{align}
\end{subequations}
Thanks to \eqref{acdc3}, \eqref{notadam}, and \eqref{acdc8}, we can construct by a classical Perron method a solution 
$$
u_\varepsilon\in C^1([0,\infty),L_2(D))\cap C([0,\infty),W_2^2(D))
$$ 
to the initial boundary value problem
\begin{align*}
\partial_t u_\varepsilon - \Delta u_\varepsilon & = - \lambda (1-\varepsilon^3) g_W(u_\varepsilon) \;\text{ in }\; (0,\infty)\times D\ , \\
u_\varepsilon & = 0 \;\text{ on }\; (0,\infty)\times \partial D\ , \\
u_\varepsilon(0) & = 0 \;\text{ in }\; D\ , 
\end{align*}
which satisfies $v_\varepsilon\le u_\varepsilon \le 0$ and $g_W(u_\ve)\le g_W(v_\ve)$ in $(0,\infty)\times D$.  It follows from \eqref{acdc3} that $u_\varepsilon\ge \varepsilon^3-1$ in $(0,\infty)\times D$ so that $u_\varepsilon$ is actually the solution to \eqref{EvEq001} with $\lambda (1-\varepsilon^3)$ instead of $\lambda$, the uniqueness being guaranteed by Theorem~\ref{EvThm005}. We then infer from Theorem~\ref{EvThm004} and Proposition~\ref{jovanotti}  that 
$$
U_{\lambda (1-\varepsilon^3)} = \inf_{t\ge 0} u_\varepsilon(t) \ge \varepsilon^3 - 1 \;\text{ in }\; D\ .
$$
However, $\lambda(1-\varepsilon^3)>\Lambda_z$ for $\varepsilon$ small enough and the just obtained lower bound contradicts the definition of $\Lambda_z$. Therefore, there is $T>0$ such that $\zeta_u(T)\not\equiv 0$ and thus $|\mathcal{C}(T)|>0$. Owing to the time monotonicity \eqref{gilmour} of the coincidence set, we have shown that $u(t)$ is zipped for $t\ge T$ and the proof is complete. 
\end{proof}

\begin{rem}
Inequality \eqref{neville} is actually valid for any energy solution to \eqref{EvEq001} with initial value $u_0\in\mathcal{A}\cap L_\infty(D)$ for $t$ large enough, thereby guaranteeing that $\zeta_u \not\equiv 0$ in $(0,t)\times D$ for such $t$. However, it is not clear whether this implies that $\zeta_{u}(t)\not\equiv 0$ for all $t$ sufficiently large.
\end{rem}

\section*{Acknowledgments}

Part of this work was done while PhL enjoyed the hospitality and support of the Institut f\"ur Angewandte Mathematik, Leibniz Universit\"at Hannover.



\end{document}